\documentclass[11pt]{amsart}
\usepackage{amsmath}
\usepackage{amsfonts}
\usepackage{amssymb}

\newcommand\eps{\varepsilon}
\newcommand\R{{\mathbb{R}}}
\newcommand\C{{\mathbb{C}}}
\newcommand\Z{{\mathbb{Z}}}

\newcommand\N{{\mathbb{N}}}

\DeclareMathOperator*{\supp}{{Supp}}
\renewcommand{\hat}[1]{\widehat{#1}}

\DeclareMathOperator*{\dist}{dist}
\DeclareMathOperator*{\diam}{diam}

\DeclareMathOperator*{\wlim}{wk-lim}

\newcommand{\jp}[1]{\langle{#1}\rangle}

\newcommand{\jpn}{\langle\nabla\rangle}
\newcommand{\qtq}[1]{\quad\text{#1}\quad}

\theoremstyle{plain}
  \newtheorem{theorem}[subsection]{Theorem}
  
  \newtheorem{proposition}[subsection]{Proposition}
  \newtheorem{lemma}[subsection]{Lemma}
  \newtheorem{corollary}[subsection]{Corollary}

\theoremstyle{remark}

\theoremstyle{definition}
  \newtheorem{definition}[subsection]{Definition}  %
\begin{document}
\title[Profile decomposition for 4th order Schr\"odinger]{Linear profile decompositions for a family of fourth order Schr\"odinger equations}
\author{Jin-Cheng Jiang}
\address{Department of Mathematics, National Tsing Hua University, Hsinchu, Taiwan 30013, R.O.C.}
\email{jcjiang@math.nthu.edu.tw}
\author{Shuanglin Shao}
\address{Department of Mathematics, University of Kansas, Lawrence, KS  66045}
\email{slshao@ku.edu}
\author{Betsy Stovall}
\address{Department of Mathematics, University of Wisconsin, Madison, WI 53706}
\email{stovall@math.wisc.edu}

\begin{abstract}
We establish linear profile decompositions for the fourth order Schr\"odinger equation and for certain fourth order perturbations of the Schr\"odinger equation, in dimensions greater than or equal to two.  We apply these results to prove dichotomy results on the existence of extremizers for the associated Stein--Tomas/Strichartz inequalities; along the way, we also obtain lower bounds for the norms of these operators.  
\end{abstract}
\date{\today}

\maketitle

\section{Introduction}

We consider the family of fourth order Schr\"odinger equations 
\begin{equation} \label{E:4NLS}
\begin{cases}
&iu_t + \Delta^2 u - \mu\Delta u = 0, \qquad \mu \geq 0, \\
&u(0) = u_0 \in L^2_x(\R^d)
\end{cases}
\end{equation}
where $u:\R \times \R^d \to \C$, $d \geq 2$.  For $\mu \neq 0$, this is the free form of the nonlinear Schr\"odinger equation with a fourth order perturbation; this equation was introduced by Karpman \cite{Karpman91} (see also \cite{KS91, K961, K962, KS00}) to study the effects of higher order dispersion in the propagation of solitary waves in plasmas.  

The main result of this article (Theorem~\ref{thm:linear-profile}) is a linear profile decomposition for the equations given by \eqref{E:4NLS}.  The theorem roughly states that, after passing to a subsequence, an $L^2_x$-bounded sequence of initial data may be decomposed as sum of asymptotically orthogonal pieces that are compact modulo symmetries, plus an error term with arbitrarily small dispersion.  We then use this result to obtain dichotomy results on the existence of extremizers to the inequalities
$$
\|(\mu+|\nabla|^2)^{\frac d{2(d+2)}} e^{it(\Delta^2-\mu\Delta)} f\|_{L^{\frac{2(d+2)}d}_{t,x}} \leq C_{d,\mu} \|f\|_{L^2_x}, 
$$
in the spirit of Shao \cite{ShaoAiry}.  Our results generalize those of Jiang, Pausader and Shao \cite{Jiang-Pausader-Shao:2010:4th-NLS}, wherein the analogous theorems were proved in the one dimensional case.  Though the results we obtain are similar, we encounter new challenges in higher dimensions.  One reason for this is that the propagator $e^{it(\Delta^2-\mu\Delta)}$ may be viewed as a Fourier extension operator, and the analysis of such operators seems to be much more difficult in dimensions greater than or equal to two.  

We are additionally motivated by recent applications of linear profile decompositions to the study of other dispersive equations, including wave \cite{Bahouri-Gerard:1999:profile-wave, KMActa08}, 
Schr\"odinger \cite{Begout-Vargas:2007:profile-schrod-higher-d, Merle-Vega:1998:profile-schrod, Carles-Keraani:2007:profile-schrod-1d, KMinvent06, KTV, KVZ}, 
KdV \cite{ShaoAiry, KKSV}, 
and Klein--Gordon \cite{KSVtams}.  

Finally, much of the argument seems amenable to an extension to more general perturbations of the Schr\"odinger equation, for instance with $\Delta^2$ replaced by $|\nabla|^\alpha$ for $\alpha>0$, but the authors have not investigated the extent to which this argument would need to be changed.  

We now turn to a brief outline of the proof.  For the remainder of this article, we let $S_\mu(t)$ denote the data-to-solution map,
$$
S_\mu(t) = e^{it(\Delta^2-\mu\Delta)}
$$
and let $D_\mu$ denote the differential operator
$$
D_\mu = \sqrt{\mu+|\nabla|^2}.
$$

Section~\ref{S:refined st} is devoted to a proof of a refinement (Proposition~\ref{P:refined Strichartz}) of the Strichartz/Stein--Tomas inequality 
\begin{equation} \label{E:ST}
\|D_\mu^{\frac d{d+2}} S_\mu(t) f\|_{L^{\frac{2(d+2)}d}_{t,x}} \lesssim \|f\|_{L^2_x}.
\end{equation}
Roughly, this result states that if the left side of this inequality is greater than a constant times the right side, the function $f$ must contain a nontrivial wave packet that is concentrated on a small `cap' on the Fourier side.  To overcome the difficulty of vanishing Gaussian curvature at the origin of the quartic surface, we reduce estimating it to an annulus $\{|\xi|\sim 1\}$, where after standard normalizations the results of Tao on bilinear paraboloid restriction can be applied. This result is stronger than the annular refinement obtained in Chae, Hong and Lee \cite[Proposition~2.3]{CHL} (and necessary for our finer-scale decomposition), and its proof strongly relies on Tao's bilinear restriction theorem for elliptic hypersurfaces  from Tao \cite{Tao:2003:paraboloid-restri}.  One challenge that we face in proving the refined Strichartz inequality for these fourth order equations (compared with wave, Schr\"odinger, or Klein--Gordon) is the lack of scale invariance when $\mu\neq 0$, coupled with the absence of any natural analogue of the Lorentz or Galilei boosts.  (The phase shifts $S_\mu(t)f \mapsto S_\mu(t)e^{i(\cdot)a}f$ will provide a rough stand-in.)

Once we have obtained this refinement, we turn in Section~\ref{S:lpd} to the proof of the linear profile decomposition.  This result, Theorem~\ref{thm:linear-profile}, follows by a familiar inductive argument.  If $(u_n)$ is an $L^2_x$-bounded sequence such that $\|D_\mu^{\frac d{d+2}}S_\mu(t)u_n\|_{L^{\frac{2(d+2)}d}_{t,x}}$ does not tend to 0,  by the refined Strichartz estimate, there exists a sequence $\{g_n^1\}$ of pseudo-symmetries (true symmetries, i.e.\ spacetime translations, composed with scalings and phase shifts) such that the sequence $(g_n^1)^{-1} u_n$ has a nonzero weak limit $\phi^1 \in L^2_x$.  The first profile is $\phi_n^1 = g_n^1 \phi^1$, and we then repeat the argument on the sequence $u_n-\phi_n^1$.  In fact, the refined Strichartz estimate gives a quantitative lower bound on the $L^2_x$ norm of $\phi^1$, and it is this that allows us to eventually show that for large $l$, the error terms $w_n^l$ are negligible.  Having established the linear profile decomposition, for our application, we need to prove that the pieces $S_\mu(t) \phi_n^j$ and $S_\mu(t)\phi_n^k$ are asymptotically orthogonal for $j \neq k$.  This is the content of Proposition~\ref{P:ST ortho}.  

Finally, in Section~\ref{S:extreme}, we apply the linear profile decomposition to prove Theorems~\ref{T:dichotomy 0} and~\ref{T:dichotomy 1}, which give lower bounds for the operator norms and dichotomy results on the existence of extremizers to \eqref{E:ST} when $\mu=0$ or $1$ (by scaling, this extends to the general case).  Very roughly, because of the asymptotic orthogonality of the profiles in the decomposition, after passing to a subsequence, an extremizing sequence to \eqref{E:ST} must contain a single profile.  After passing to a subsequence, there are three possibilities:  compactness, convergence to a free Schr\"odinger wave, or convergence to a free $S_0$ wave.  In the first case, extremizers exist.  In the latter two, extremizers may fail to exist, but these cases give us the desired lower bounds on the operator norms.  

\textbf{Acknowledgements.}  The research of the first author was supported by National Science
Council Grant NSC100-2115-M-007-009-MY2.  The second author was supported by NSF DMS-1160981 and KU 2016-2017 general research fund. The third author was supported by NSF DMS-0902667 and 1266336.


\section{The refined Strichartz inequality}\label{S:refined st}


Fix $d \geq 2$ and $\mu \geq 0$.  For the remainder of the article, let $\phi:\R^d \to [0,1]$ be a smooth, radial, decreasing bump function with $\phi \equiv 1$ on $\{|\xi| \leq 1\}$ and $\phi \equiv 0$ on $\{|\xi|\geq 2\}$.  

In this section, we prove the refined Strichartz inequality, which is crucial to our profile decomposition.  Before stating it, we need a little notation.  

\begin{definition} \label{D:caps}
A $\mu$-\textit{cap} is a ball $\kappa = \{\xi:|\xi-\xi_0| < r\}$, for some $\xi_0 \in \R^d$ and $r>0$ satisfying $8r \leq |\xi_0| + \sqrt{\mu}$.  
\end{definition}

Given a cap $\kappa$ with center $\xi_0$ and radius $r$, we define an associated cutoff
$$
\phi_\kappa(\xi) = \phi\bigl(\frac{\xi-\xi_0}r\bigr).
$$
Given $f$ we denote by $f_\kappa$ the function whose Fourier transform is given by
$$
\widehat{f_\kappa} = \phi_\kappa \widehat{f}.
$$

With this notation in place, our refined Strichartz inequality is the following.

\begin{proposition}[Refined Strichartz] \label{P:refined Strichartz}
Let $q = \frac{2(d^2+3d+1)}{d^2}$ and $\theta = \frac2{(d+2)^2}$.  If $f \in L^2_x(\R^d)$,  then
\begin{equation}\label{E:refined Strichartz}
\|S_\mu(t)D_\mu^{\frac{d}{d+2}}f\|_{L^{\frac{2(d+2)}d}_{t,x}} \lesssim \bigl(\sup_\kappa |\kappa|^{\frac{d+2}{dq}-\frac12}\|S_\mu(t)D_\mu^{\frac2q} f_\kappa\|_{L^q_{t,x}})^{\theta} \|f\|_{L^2_x}^{1-\theta}.
\end{equation}
Here the supremum is taken over all $\mu$-caps $\kappa$ as in Definition~\ref{D:caps}, and the implicit constant depends only on $d$.
\end{proposition}

Propositions of this kind have appeared in many places in the literature, and the outline we follow is a familiar one (cf.\ \cite{Begout-Vargas:2007:profile-schrod-higher-d, Killip-Visan:2008:clay-lecture-notes}).  The main new ingredient here is the parameter $\mu$.  When $\mu=0$, the graph of the function $\mu|\xi|^2+|\xi|^4$ has vanishing curvature at 0, while in the case $\mu >0$, the curvature never vanishes, but the scaling symmetry is broken.  To deal with these issues, we begin by proving a refined Strichartz estimate associated to the decoupling of dyadic frequency annuli.  By this and scaling, it suffices to establish a refinement of the Stein--Tomas inequality for the Fourier extension operator associated to the surfaces $\{(|\xi|^2+|\xi|^4,\xi) : |\xi| \lesssim 1\}$ and $\{(\eps|\xi|^2 + |\xi|^4,\xi) : |\xi| \sim 1\}$, where $0 \leq \eps \lesssim 1$.  These surfaces are uniformly well-curved; indeed, they are elliptic on small (but uniform) frequency scales.  We can thus apply Tao's bilinear restriction theorem, and use methods developed in the context of the linear profile decomposition for Schr\"odinger to obtain refined estimates on these frequency localized regions.  Finally, it is a simple matter to undo the scaling and glue the pieces back together, thereby obtaining \eqref{E:refined Strichartz}.  

We begin by noting that it suffices to prove Proposition~\ref{P:refined Strichartz} when $\mu=0,1$.  Indeed, if $\mu>0$, a simple computation shows that
$$
S_\mu(t)f = [S_1(\mu^2 t) f(\tfrac\cdot{\sqrt\mu})]({\sqrt\mu}\:\cdot\:), \qtq{and} D_\mu f = [\sqrt\mu \, D_1 f(\tfrac\cdot{\sqrt\mu})]({\sqrt\mu}\:\cdot\:).
$$
Furthermore, if $\kappa = \{\xi : |\xi-\xi_0|<r\}$ is a 1-cap, $\sqrt\mu \, \kappa = \{\xi : |\xi - \sqrt\mu \,\xi_0| < \sqrt\mu\, r\}$ is a $\mu$-cap.  Thus Proposition~\ref{P:refined Strichartz} for any $\mu > 0$ follows from Proposition~\ref{P:refined Strichartz} in the case $\mu=1$ by scaling.  For the remainder of this section, we will consider only the cases $\mu=0,1$.  

We will employ two different Littlewood--Paley decompositions, depending on whether $\mu=0$ or $\mu = 1$.  Define 
\begin{gather*}
\psi_0^1(\xi) = \phi(2\xi), \qquad \psi_N^1(\xi) = \phi(\tfrac\xi N)-\phi(\tfrac{2\xi}N), \qtq{for} N \in 2^{\N},\\
\psi_N^0(\xi) = \phi(\tfrac\xi N) - \phi(\tfrac{2\xi}N), \qtq{for} N \in 2^\Z.
\end{gather*}
Regardless of the value of $\mu$, the $\psi_N^\mu$ form a partition of unity.  We define the Littlewood--Paley projections $P_N^\mu$ by $\widehat{f_N^\mu} = \widehat{P_N^\mu f} = \psi_N^\mu \widehat{f}$, $\mu=0,1$.

\begin{lemma} \label{L:annular decoupling}
If $f \in L^2_x$, then
\begin{equation} \label{E:annular decoupling}
\|S_\mu(t)D_\mu^{\frac{d}{d+2}}f\|_{L^{\frac{2(d+2)}d}_{t,x}} \lesssim \sup_N\|S_\mu(t)D_\mu^{\frac{d}{d+2}}f^\mu_N\|_{L^{\frac{2(d+2)}d}_{t,x}}^{\frac2{d+2}}\|f\|_{L^2_x}^{\frac{d}{d+2}},
\end{equation}
for $\mu=0,1$.  Here, the supremum is taken over frequencies $N \in \{0\} \cup 2^{\N}$ or $N \in 2^\Z$, depending on the value of $\mu$.  
\end{lemma}

In the proof of Lemma~\ref{L:annular decoupling}, we will use the following Strichartz estimates, which may be proved using the methods of stationary phase together with the main theorem of Keel and Tao \cite{KeelTao}.  See Pausader \cite{Pausader:2007} for further details.  

\begin{proposition}
Let $\mu \geq 0$.  For all $(q,r)$ satisfying $2 \leq q,r \leq \infty$, $(q,r) \neq (2,\infty)$, and $\frac2q + \frac dr = \frac d2$, 
\begin{equation} \label{E:Strichartz}
\| D_{\mu} ^{\frac2q}S_\mu(t) f\|_{L^q_tL^r_x(\R \times \R^d)} \lesssim \|f\|_{L^2_x(\R^d)}.
\end{equation}
The implicit constant may be taken to depend only on $d,q,r$ and, in particular, may be chosen independently of $\mu$.  
\end{proposition}

\begin{proof}
Let $d = 2$.  By the Littlewood--Payley square function estimate, simple arithmetic, H\"older's inequality, and the Strichartz inequality, for $\mu=0,1$,
\begin{align*} 
&\|S_\mu(t)D_\mu^{\frac12}f\|_{L^4_{t,x}}^4
 \sim \iint (\sum_M |S_\mu(t)D_\mu^{\frac12}f_M^\mu|^2)(\sum_N |S_\mu(t)D_\mu^{\frac12}f_N^\mu|^2)\, dx\, dt\\
&\qquad \sim \sum_{M \geq N} \|(S_\mu(t)D_\mu^{\frac12}f_M^\mu)(S_\mu(t)D_\mu^{\frac12}f_N^\mu)\|_{L^2_{t,x}}^2\\
& \qquad\lesssim \sum_{M \geq N} \|S_\mu(t)D_\mu^{\frac12}f_M^\mu\|_{L^4_{t,x}}\|S_\mu(t)D_\mu^{\frac12}f_M^\mu\|_{L^3_tL^6_x}
\|S_\mu(t)D_\mu^{\frac12}f_N^\mu\|_{L^4_{t,x}}\|S_\mu(t)D_\mu^{\frac12}f_N^\mu\|_{L^6_tL^3_x}\\ 
&\qquad\lesssim \sup_K\|S_\mu(t)D_\mu^{\frac12} f_K^\mu\|_{L^4_{t,x}}^2 \sum_{M \geq N}\|D_\mu^{-\frac16}f_M^\mu\|_{L^2_x}\|D_\mu^{\frac16}f_N^\mu\|_{L^2_x}.  
\end{align*}
If $\mu=0$, this implies by Plancherel that
$$
\|S_\mu(t)D_\mu^{\frac12}f\|_{L^4_{t,x}}^4 \lesssim \sup_K\|S_\mu(t)D_\mu^{\frac12}f_K^0\|_{L^4_{t,x}}^2 \sum_{M \geq N}M^{-\frac16}N^{\frac16}\|f_M^0\|_{L^2_x}\|f_N^0\|_{L^2_x},
$$
while if $\mu=1$, 
$$
\|S_\mu(t)D_\mu^{\frac12}f^1\|_{L^4_{t,x}}^4 \lesssim \sup_K\|S_\mu(t)D_\mu^{\frac12}f_K^1\|_{L^4_{t,x}}^2 \sum_{M \geq N}\jp{M}^{-\frac16}\jp{N}^{\frac16}\|f_M^1\|_{L^2_x}\|f_N^1\|_{L^2_x}.
$$
In either case, \eqref{E:annular decoupling} follows from an application of Schur's test.

If $d>2$, then by the square function estimate, the triangle inequality, H\"older's inequality, and the endpoint Strichartz inequality, for $\mu=0,1$,
\begin{align*}
&\|S_\mu(t)D_\mu^{\frac{d}{d+2}} f\|_{L^{\frac{2(d+2)}d}_{t,x}}^{\frac{2(d+2)}d}
\sim \iint \bigl(\sum_M|S_\mu(t)D_\mu^{\frac{d}{d+2}}f_M^\mu|^2\bigr)^{\frac{d+2}{2d}}\bigl(\sum_N|S_\mu(t)D_\mu^{\frac{d}{d+2}}f_N^\mu|^2\bigr)^{\frac{d+2}{2d}}\, dx\, dt\\
&\qquad \lesssim \sum_{M \geq N} \|\bigl(S_\mu(t)D_\mu^{\frac{d}{d+2}}f_M^\mu\bigr)\bigl(S_\mu(t)D_\mu^{\frac{d}{d+2}}f_N^\mu\bigr)\|_{L^{\frac{d+2}{d}}_{t,x}}^{\frac{d+2}{d}}\\
& \qquad \lesssim \sum_{M \geq N} \|S_\mu(t)D_\mu^{\frac{d}{d+2}}f_M^\mu\|_{L^{\frac{2(d+2)}{d}}_{t,x}}\|S_\mu(t)D_\mu^{\frac{d}{d+2}}f_M^\mu\|_{L^2_tL^{\frac{2d}{d-2}}_x}^{\frac2d}\\
&\qquad \qquad \qquad \times \|S_\mu(t)D_\mu^{\frac{d}{d+2}}f_N^\mu\|_{L^{\frac{2(d+2)}{d}}_{t,x}}\|S_\mu(t)D_\mu^{\frac{d}{d+2}}f_N^\mu\|_{L^{\frac{2(d+2)}{d-2}}_tL^{\frac{2d(d+2)}{d^2+4}}_x}^{\frac2d}\\
&\qquad \lesssim \sup_K \|S_\mu(t)D_\mu^{\frac{d}{d+2}}f_K^\mu\|_{L^{\frac{2(d+2)}{d}}_{t,x}}^{\frac4d}
\sum_{M \geq N} \|f_M^\mu\|_{L^2_x}^{1-\frac2d}\|D_\mu^{-\frac{2}{d+2}}f_M^\mu\|_{L^2_x}^{\frac2d}\|f_N^\mu\|_{L^2_x}^{1-\frac2d} \|D_\mu^{\frac2{d+2}}f_N^\mu\|_{L^2_x}^{\frac2d}.
\end{align*}
By Plancherel, if $\mu=0$, 
$$
\|S_\mu(t)D_\mu^{\frac{d}{d+2}} f\|_{L^{\frac{2(d+2)}d}_{t,x}}^{\frac{2(d+2)}d}\lesssim \sup_K \|S_\mu(t)D_\mu^{\frac{d}{d+2}}f_K^0\|_{L^{\frac{2(d+2)}{d}}_{t,x}}^{\frac4d}\sum_{M \geq N} M^{-\frac4{d(d+2)}}N^{\frac4{d(d+2)}}\|f_M^0\|_{L^2_x}\|f_N^0\|_{L^2_x},
$$
while if $\mu=1$,
$$
\|S_\mu(t)D_\mu^{\frac{d}{d+2}} f\|_{L^{\frac{2(d+2)}d}_{t,x}}^{\frac{2(d+2)}d}\lesssim \sup_K \|S_\mu(t)D_\mu^{\frac{d}{d+2}}f_K^1\|_{L^{\frac{2(d+2)}{d}}_{t,x}}^{\frac4d}\sum_{M \geq N} \jp{M}^{-\frac4{d(d+2)}}\jp{N}^{\frac4{d(d+2)}}\|f_M^1\|_{L^2_x}\|f_N^1\|_{L^2_x}.
$$
As in the $d=2$ case, \eqref{E:refined Strichartz} follows by Schur's test. 
\end{proof}

We now rescale one more time, to frequencies $|\xi| \sim 1$.  We claim that the proposition follows from the following.

\begin{lemma} \label{L:scale 1 refinement}
Let $q = \tfrac{2(d^2+3d+1)}{d^2}$.  For $f \in L^2_x(\R^d)$, at scale 1 we have
\begin{equation} \label{E:scale 1}
\|S_\eps(t) P_1^0 f\|_{L^{\frac{2(d+2)}d}_{t,x}} \lesssim \bigl(\sup_\kappa|\kappa|^{\frac{d+2}{dq}-\frac12}\|S_\eps(t)f_\kappa\|_{L^q_{t,x}}\bigr)^{\frac1{d+2}}\|f\|_{L^2_x}^{\frac{d+1}{d+2}}, \quad 0 \leq \eps \leq 1,
\end{equation}
where the supremum is taken over caps $\kappa = \{\xi:|\xi-\xi_0| < r\}$, with $\tfrac12 \leq |\xi_0| \leq 2$ and $r < \tfrac1{16}$.  Additionally, at scale 0,
\begin{equation} \label{E:scale 0}
\|S_1(t) P_0^1 f\|_{L^{\frac{2(d+2)}d}_{t,x}} \lesssim \bigl(\sup_\kappa |\kappa|^{\frac{d+2}{dq}-\frac12}\|S_1(t) f_\kappa\|_{L^q_{t,x}}\bigr)^{\frac1{d+2}} \|f\|_{L^2_x}^{\frac{d+1}{d+2}},
\end{equation}
where the supremum is taken over caps $\kappa = \{\xi : |\xi-\xi_0| < r\}$, with $|\xi_0|\leq 1$ and $r < \tfrac1{16}$.  
\end{lemma}

Assuming the lemma for the moment, we complete the proof of Proposition~\ref{P:refined Strichartz}.  

\begin{proof}[Proof of Proposition~\ref{P:refined Strichartz}]
By scaling, \eqref{E:scale 1} implies that 
\begin{equation} \label{E:scale N}
N^{\frac d{d+2}}\|S_{N\eps}(t)P_N^0 f\|_{L^{\frac{2(d+2)}d}_{t,x}} \lesssim \bigl(\sup_\kappa|\kappa|^{\frac{d+2}{dq}-\frac12}N^{\frac2q}\|S_{N\eps}(t) f_\kappa\|_{L^q_{t,x}}\bigr)^{\frac1{d+2}}\|f\|_{L^2_x}^{\frac{d+1}{d+2}},
\end{equation}
for $0 \leq \eps \lesssim 1$ and $N \in 2^\Z$, where the supremum is taken over caps $\kappa = \{|\xi-\xi_0|<r\}$, where $\tfrac N4 \leq |\xi_0| \leq 4N$ and $r < \tfrac N{32}$.  We note that if $\mu=0$ or $\mu=1$ and $N \geq 1$, these are $\mu$-caps.  

In the case $\mu=0$, the refined Strichartz estimate \eqref{E:refined Strichartz} just follows from \eqref{E:annular decoupling} followed by \eqref{E:scale N} (with $\eps=0$) and Bernstein's inequality.  

In the case $\mu=1$, we first apply \eqref{E:annular decoupling}.  Then we use Bernstein's inequality together with \eqref{E:scale 0} for the scale 0 term and \eqref{E:scale N} with $\eps = \tfrac1N$ for the scale $N$ term.  (Note that if $N \in 2^N$, $P_N^0 = P_N^1$.)
\end{proof}

The proof of Lemma~\ref{L:scale 1 refinement} reduces to Lemma \ref{L:scale 1 refined Strichartz} . The reduction is similar as in Guth \cite[p.~386]{Guth:2016:restriction-polynomial-partition}. We consider a function with Fourier support in a small cap in $\{|\xi|\sim 1\}$. 

Fix $\tfrac14 < |\xi_0| < 4$ and define
$$
\sigma_{\eps,\xi_0} = \{(\eps|\xi|^2 + |\xi|^4,\xi) : |\xi-\xi_0| \leq 3 c_d\}.
$$
Suppose that the Fourier transform $\hat{f}$ is supported on $\sigma_{\eps,\xi_0}$. We consider 4th order Schr\"odinger operator. 
Let $\phi(\xi)= |\xi|^4+\eps|\xi|^2$.  
$$ S_{\eps,\xi_0}(t)f(x)=\int_{\sigma_{\eps,\xi_0}} e^{ix\cdot\xi+it(|\xi|^4+\eps|\xi|^2)} \hat{f}(\xi) d\xi=:e^{it\phi(|\nabla|)}f(x),$$
where $x\cdot \xi$ or $x\xi$ denotes the inner product of $x$ and $\xi$. Then by translation invariance, 
$$ S_{\eps,\xi_0}(t)f(x)=e^{ix\cdot\xi_0+it\phi(\xi_0)}\int_{\sigma_{\eps,\xi_0}} e^{i(x-t\nabla \phi(\xi_0))\cdot(\xi-\xi_0)+it\left(\phi(\xi)-\phi(\xi_0)-\nabla \phi(\xi_0)\cdot(\xi-\xi_0)\right)} \hat{f}(\xi) d\xi.$$
Let $h(\xi) =\phi(\xi)-\phi(\xi_0)-\nabla \phi(\xi_0)\cdot (\xi-\xi_0) $. Then 
\begin{equation}\label{galilean_1}\|S_{\eps,\xi_0}(t)f(x)\|_{L_{t,x}^{\frac {2(d+2)}{d}}} = \|e^{ith(|\nabla|)}f\|_{L_{t,x}^{\frac {2(d+2)}{d}}}
\end{equation}
The function $h$ satisfies that $h(\xi_0)=0$ and $\nabla h(\xi_0) =0$. Recall that the elliptic condition in Guth \cite[Condition 2.1]{Guth:2016:restriction-polynomial-partition}, see also Tao,~Vargas and ~Vega \cite{Tao-Vargas-Vega:1998:bilinear-restri-kakeya}. Suppose that $S\subset \mathbb{R}^3$ is a smooth compact surface given as the graph of a function $h:\, B(0,1) \to \mathbb{R}$ which satisfies the following conditions for some large $L$: 

\begin{itemize}
\item $0<\frac 12 \le \nabla^2 h \le 2$, namely, all the eigenvalues of $\nabla^2 h$ is comparable to $1.$
\item $0= h(0)$, $ \nabla h(0)=0$.
\item $h$ is $C^L$, and 
\item for $3\le l\le L $, $\|\nabla^l h\|_{C^0}\le 10^{-9}. $
\end{itemize}

The function $h$ is not in elliptic type. However since the Strichartz inequality is scaling-invariant, $h$ can be normalized to satisfy the elliptic condition above by Taylor's expansion, change of variables, and parabolic scalings, see for instance Guth \cite[p. 386]{Guth:2016:restriction-polynomial-partition}.  We remark that it may necessarily change the value of $c_d$. Then the remarks in Tao \cite[Section 9]{Tao:2003:paraboloid-restri} applies.

\begin{theorem}[\cite{Tao:2003:paraboloid-restri}] \label{T:tao bilin} Let $h$ be as above. Then the operator $\mathcal{E}_h$ defined by 
$$
\mathcal{E}_hf (t,x) = \int_{\{|\xi| \lesssim 1\}} e^{i(t,x) \cdot (h(\xi),\xi)}\hat f(\xi)\, d\xi
$$
satisfies the following.  Let $\rho_1,\rho_2 \subset \{|\xi| \lesssim 1\}$ satisfy
$$
\diam(\rho_1) \sim \diam(\rho_2) \sim \dist(\rho_1,\rho_2) \sim 1.
$$
If $f_1,f_2$ are $L^2_x$ functions whose Fourier transforms are supported on $\rho_1,\rho_2$, respectively, then
$$
\bigl\|\bigl(\mathcal{E}_h f_1)\bigl(\mathcal{E}_h f_2)\bigr\|_{L^p_{t,x}} \lesssim \|f_1\|_{L^2_x}\|f_2\|_{L^2_x}, \qquad p > \tfrac{d+3}{d+1},
$$
where the implicit constant depends only on $d,p$.
\end{theorem}

Taking advantage of the symmetries of the Fourier transform, we have the following corollary. 

\begin{corollary} \label{C:small ball bilin}
For $c_d$ sufficiently small, if $\kappa_1$ and $\kappa_2$ are two caps contained in $\{\frac14 < |\xi| < \frac32\}$ satisfying
$$
\diam(\kappa_1) \sim \diam(\kappa_2) \sim \dist(\kappa_1,\kappa_2) \sim r < c_d,
$$
then for any $f \in L^2_x$ and $p > \frac{d+3}{d+1}$,
$$
\|\bigl(\mathcal{E}_hf_{\kappa_1}\bigr)\bigl(\mathcal{E}_h f_{\kappa_2}\bigr)\|_{L^p_{t,x}} \lesssim_{d,p} r^{d-\frac{d+2}p}\|f_{\kappa_1}\|_{L^2_x}\|f_{\kappa_2}\|_{L^2_x}, \qquad p > \frac{d+3}{d+1}.
$$
\end{corollary}

Then the following refined Strichartz estimate follows similarly as in B\'egout and Vargas \cite{Begout-Vargas:2007:profile-schrod-higher-d}, see also Proposition~4.23 of Killip and Visan \cite{Killip-Visan:2008:clay-lecture-notes}.

\begin{lemma} \label{L:scale 1 refined Strichartz}
Let $q = \frac{2(d^2+3d+1)}{d^2}$.  If $f \in L^2_x$ and Fourier supported in a cap $\sigma_{\eps, \xi_0}$ of size $c_d$ in the annulus $\{|\xi|\sim 1\}$, then
\begin{equation} \label{E:scale 1 refined Strichartz}
\|S_{\eps,\xi_0}(t) f\|_{L^{\frac{2(d+2)}{d}}_{t,x}}\lesssim \bigl(\sup_\kappa |\kappa|^{\frac{d+2}{dq}-\frac12}\|S_{\eps,\xi_0}(t)f_\kappa\|_{L^q_{t,x}}\bigr)^{\frac1{d+2}}\|f\|_{L^2_x}^{\frac{d+1}{d+2}},
\end{equation}
where the supremum may be taken over caps $\kappa$ with
\begin{equation} \label{E:def T1}
\kappa \subset \{ |\xi|\sim 1\} \qtq{and} \diam(\kappa) < .01.
\end{equation}
\end{lemma}

Now we are ready to prove Lemma~\ref{L:scale 1 refinement} by a decomposition of unit of $\{|\xi|\sim 1\}$ into many caps with diameters comparable to $c_d$. The cardinality of these caps depends on the spatial dimension only. More precisely, we recall that $\widehat {P_1^0 f} = \psi_1^0 \widehat f$, where $\psi_1^0$ is supported on $\{\tfrac12 \leq |\xi| \leq 2\}$.  There exists a finite decomposition 
\begin{equation} \label{E:decompose psi}
\psi_1^0 = \sum_{j=1}^{C_d} \psi_{1,j}, 
\end{equation}
with $\psi_{1,j} = \psi_1^0\phi_j$, for smooth bump functions $\phi_j:\R^d \to [0,1]$ whose supports have diameter equal to a small dimensional constant $c_d$' $C_d>0$ depends on $d$ only.  

By \eqref{E:decompose psi}, for any $f \in L^2_x$,
\begin{equation} \label{E:decompose f1}
P_1^0 f = \sum_{j=1}^{C_d} f_{1,j}, \qtq{where} \widehat{f_{1,j}} = \psi_{1,j}\widehat{f}, \quad i=1,2.
\end{equation}
By \eqref{E:decompose f1} and the triangle inequality,
\begin{equation} \label{E:small ball decoupling}
\|S_\eps(t) f_1\|_{L^{\frac{2(d+2)}d}_{t,x}} \leq \sum_{j=1}^{C_d}\|S_\eps(t) f_{1,j}\|_{L^{\frac{2(d+2)}d}_{t,x}} \leq C_d \max_j \|S_\eps(t) f_{1,j}\|_{L^{\frac{2(d+2)}d}_{t,x}}.
\end{equation}

\begin{proof}[The proof of Lemma~\ref{L:scale 1 refinement}] 
Since $|\psi_{1,j}| \leq 1$, by Plancherel, $\|f_{1,j}\|_{L^2_x} \leq \|f\|_{L^2_x}$.  Furthermore, we have $S_\eps(t)[(f_{1,j})_\kappa] = (S_\eps(t) f_\kappa)_{1,j}$.  By construction, $\|\check\psi_{1,j}\|_{L^1_x} \lesssim 1$, so by Young's convolution inequality (in the $x$-variable), $\|S_\eps(t)(f_{1,j})_\kappa\|_{L^q_{t,x}} \lesssim \|S_\eps(t)f_\kappa\|_{L^q_{t,x}}$.  So finally, by combining \eqref{E:small ball decoupling} and \eqref{E:scale 1 refined Strichartz}, we obtain \eqref{E:scale 1}.  The inequality \eqref{E:scale 0} follows in a similar manner (but is simpler because there is no $\eps$).  Therefore Lemma~\ref{L:scale 1 refinement} is proved. 
\end{proof}


\section{Linear profile decomposition} \label{S:lpd}


\begin{theorem}[Linear profile decomposition]\label{thm:linear-profile}
Let $(u_n)_{n\ge 1}$ be a bounded sequence of $L^2_x$ functions.  After passing to a subsequence, the following hold.  For each $j \geq 1$ there exist $\phi^j \in L^2_x$, a sequence of parameters $(h_n^j,\xi_n^j,x_n^j,t_n^j) \in (0,\infty) \times \R^d \times \R^d \times \R$, and a sequence of $L^2_x$ errors $(w_n^l)_{l,n \geq 1}$ such that for each $l \geq 1$, 
\begin{equation}\label{eq:profile-decomposition}
u_n=\sum_{j=1}^l S_\mu(-t_n^j)g^j_n[e^{i(\cdot)h^j_n\xi^j_n}\phi^j]+w^l_n,
\end{equation}
where $g^j_n(\phi):=\frac{1}{(h^j_n)^{d/2}}\phi(\frac{x-x^j_n}{h^j_n})$, and for each $j$, either $|h_n^j \xi_n^j| \to \infty$ or $\xi_n^j \equiv 0$.  The errors satisfy
\begin{equation}\label{eq:error-term}
\limsup_{l\to\infty}\limsup_{n\to\infty}\|D_\mu^{1/2}S_\mu(t)w^l_n\|_{L^{\frac{2(d+2)}d}_{t,x}(\mathbb{R}
\times\mathbb{R}^d)}=0.
\end{equation} 
For $j \neq k$, $(h_n^j,\xi_n^j,x_n^j,t_n^j)_{n\ge 1}$ and $(h_n^k,\xi_n^k,x_n^k,t_n^k)_{n\ge 1}$ are pairwise orthogonal in the sense that 
\begin{equation} \label{E:ortho}
\begin{aligned}
&\lim_{n \to \infty} \bigl( \frac{h_n^j}{h_n^k}+\frac{h_n^k}{h_n^j} + |h_n^j\xi_n^j-h_n^k\xi_n^k| + \frac{|t_n^k-t_n^j|}{(h_n^j)^4} + \frac{|t_n^k-t_n^j|(\mu+2|\xi_n^j|^2)}{(h_n^j)^2}  \\
&\qquad + \frac{|x_n^j-x_n^k+2(t_n^j-t_n^k)(2|\xi_n^j|^2+\mu)\xi_n^j|}{h_n^j}\bigr) = \infty.
\end{aligned}
\end{equation}
Furthermore, if $\lim_{n \to \infty}(\frac{h_n^j}{h_n^k} + \frac{h_n^k}{h_n^j}) \neq \infty$, then $h_n^j = h_n^k$ for all $n$, and if $\lim_{n \to \infty}|h_n^k\xi_n^k - h_n^j \xi_n^j| \neq \infty$, then $h_n^j \xi_n^j = h_n^k \xi_n^k$ for all $n$.  Finally, for each $l \geq 1$,
\begin{equation} \label{E:L2 decoup}
\lim_{n \to \infty}\bigl(\|u_n\|_{L^2_x}^2 - (\sum_{j=1}^l \|\phi^j\|_{L^2_x}^2 + \|w_n^l\|_{L^2_x}^2)\bigr) = 0
\end{equation} 
\end{theorem}

The orthogonality condition of parameters implies the following.

\begin{proposition}[Orthogonality of profiles]\label{P:ST ortho}
For $j\neq k$, along the subsequence satisfying \eqref{E:ortho},
\begin{equation}\label{E:ST ortho}
\lim_{n\to\infty}\left\|\bigl(D_\mu^{\frac d{d+2}}S_\mu(t)S_\mu(-t_n^j) g_n^j[e^{i(\cdot)h_n^j\xi_n^j}\phi^j]\bigr)\bigl( D_\mu^{\frac d{d+2}}S_\mu(t)S_\mu(-t_n^k)g_n^k[e^{i(\cdot)h_n^k\xi_n^k}\phi^k]\bigr)\right\|_{L^{\frac{d+2}d}_{t,x}}=0.
\end{equation}
Thus by \eqref{eq:error-term},
\begin{equation} \label{E:Lp decoup}
\lim_{l \to \infty} \limsup_{n \to \infty} \bigl|\|D_\mu^{\frac d{d+2}} S_\mu(t)u_n\|_{L^{\frac{2(d+2)}d}_{t,x}}^{\frac{2(d+2)}d} - \sum_{j=1}^l \|D_\mu^{\frac{d}{d+2}}S_\mu(t-t_n^j) g_n^j [e^{i(\cdot)h_n^j\xi_n^j}\phi^j]\|_{L^{\frac{2(d+2)}d}_{t,x}}^{\frac{2(d+2)}d}\bigr| = 0.
\end{equation}
\end{proposition}

We begin with the linear profile decomposition.  

\begin{proof}[Proof of Theorem~\ref{thm:linear-profile}] The proof follows a familiar outline (cf.\ \cite{Bahouri-Gerard:1999:profile-wave, Carles-Keraani:2007:profile-schrod-1d, Jiang-Pausader-Shao:2010:4th-NLS}).  We construct the linear profile decomposition inductively.  At each stage we will pass to a further subsequence, but to avoid a proliferation of sub- and superscripts, we will denote each subsequence by $(u_n)$.  For each $n,j$, let $B_n^j$ denote the $L^2_x$ isometry
$$
B_n^j = S_\mu(-t_n^j)g_n^j e^{i(\cdot)h_n^j\xi_n^j}.
$$

Set $w_n^0 = u_n$ and assume that for some $l \geq 0$ we have found a subsequence of $(u_n)$, $L^2_x$ functions $\phi^j$, and sequences $(h_n^j, \xi_n^j,x_n^j,t_n^j)$, $1 \leq j \leq l$ such that for all $1 \leq k \leq l$, we have the following: \eqref{E:L2 decoup} holds (with $k$ in place of $l$); \eqref{E:ortho} holds for all $k \neq j \in \{1,\ldots,l\}$; either $|h_n^k\xi_n^k| \to \infty$ or $\xi_n^k \equiv 0$; 
\begin{gather} \label{E:wnk}
u_n = \sum_{j=1}^k B_n^j\phi^j + w_n^k; \\
\label{E:phi^j}
(B_n^k)^{-1}w_n^{k-1} \rightharpoonup \phi^k, \: \text{weakly in $L^2_x$}.
\end{gather}

Define
$$
A^l = \limsup_{n \to \infty} \|w_n^l\|_{L^2_x}, \qquad \eps^l = \limsup_{n \to \infty} \|D_\mu^{\frac d{d+2}}S_\mu(t)w_n^l\|_{L^{\frac{2(d+2)}d}_{t,x}}.
$$
By the Strichartz inequality, $\eps^l \lesssim A^l$.  If $\eps^l = 0$, then we are done, so we may assume that $\eps^l > 0$.  Passing to a subsequence, we may assume that $\|w_n^l\|_{L^2_x} \sim A^l$ and $\|D_\mu^{\frac d{d+2}}S_\mu(t)w_n^l\|_{L^{\frac{2(d+2)}d}_{t,x}} \sim \eps^l$, for all $n$.  

By Proposition~\ref{P:refined Strichartz} and a little arithmetic, there exists a sequence $(\kappa_n)$ of $\mu$-caps such that for all sufficiently large $n$,
\begin{equation} \label{E:kappan}
A^l \bigl(\tfrac{\eps^l}{A^l}\bigr)^{\frac1\theta} \lesssim |\kappa_n|^{\frac{d+2}{dq}-\frac12}\|S_\mu(t)D_\mu^{\frac 2q}(w_n^l)_{\kappa_n}\|_{L^q_{t,x}}.
\end{equation}
By H\"older's inequality, Bernstein's inequality (since $(w_n^l)_{\kappa_n}$ is frequency localized), and Young's convolution inequality (since $\|\widehat{\phi_\kappa}\|_{L^1_x} \sim 1$), 
\begin{align*}
\|S_\mu(t)D_\mu^{\frac 2q}(w_n^l)_{\kappa_n}\|_{L^q_{t,x}} 
&\lesssim \|S_\mu(t)D_\mu^{\frac d{d+2}}(w_n^l)_{\kappa_n}\|_{L^{\frac{2(d+2)}d}_{t,x}}^{\frac{2(d+2)}{dq}}\|S_\mu(t)(w_n^l)_{\kappa_n}\|_{L^\infty_{t,x}}^{1-\frac{2(d+2)}{dq}}\\
&\lesssim (A^l)^{\frac{2(d+2)}{dq}} \|S_\mu(t)(w_n^l)_{\kappa_n}\|_{L^\infty_{t,x}}^{1-\frac{2(d+2)}{dq}}.
\end{align*}
Combining this with \eqref{E:kappan} and the fact that $(w_n^l)_{\kappa_n}$ is smooth (since it has compact Fourier support), there exist parameters $(x_n,t_n)$ such that
\begin{equation} \label{E:xntn}
A^l \bigl(\tfrac{\eps^l}{A^l}\bigr)^{\theta'} \lesssim |\kappa_n|^{-\frac12}|S_\mu(t_n)(w_n^l)_{\kappa_n}(x_n)|,
\end{equation} where $\theta'=\frac 1\theta$.

Write $\kappa_n = \{\xi: |\xi-\xi_n| < h_n^{-1}\}$, and set
$$
(h_n^{l+1},\xi_n^{l+1},x_n^{l+1},t_n^{l+1}) = (h_n,\xi_n,x_n,t_n).
$$
The sequence 
$$
(B_n^{l+1})^{-1}w_n^l = e^{-ixh_n\xi_n}g_n^{-1}S_\mu(t_n) w_n^l
$$
is bounded in $L^2_x$, so after passing to a subsequence, we may extract a weak limit; say
\begin{equation} \label{E:phi l+1}
(B_n^{l+1})^{-1} w_n^l \rightharpoonup \phi^{l+1}, \quad \text{weakly in $L^2_x$}.
\end{equation}
Thus we have
\begin{equation} \label{E:wnl decoup}
\begin{aligned}
&\lim_{n \to \infty} \bigl(\|w_n^l\|_{L^2_x}^2 - \|w_n^l -B_n^{l+1}\phi^{l+1}\|_{L^2_x}^2 - \|\phi^{l+1}\|_{L^2_x}^2\bigr)\\
&\qquad 
= \lim_{n \to \infty} \bigl(2\bigl\langle (B_n^{l+1})^{-1}w_n^l,\phi^{l+1}\bigr\rangle - 2\|\phi^{l+1}\|_{L^2_x}^2\bigr) = 0,
\end{aligned}
\end{equation}
which implies that \eqref{E:L2 decoup} holds with $l$ replaced by $l+1$ and
$$
w_n^{l+1} = w_n^l - B_n^{l+1}\phi^{l+1}.
$$

In addition, by \eqref{E:xntn}, the definition of $(w_n^l)_{\kappa_n}$, and a quick computation, 
\begin{align*}
A^l\bigl(\tfrac{\eps^l}{A^l})^{\theta'} &\lesssim h_n^{-\frac d2}|\int e^{ix_n\xi}e^{it_n(|\xi|^4+\mu|\xi|^2)}\phi(h_n(\xi-\xi_n)) \widehat{w_n^l}(\xi)\, d\xi | \\
&= |\langle e^{-i(\cdot)h_n \xi_n} g_n^{-1}S_\mu(t_n) w_n^l, \check\phi\rangle|.
\end{align*}
Taking the limit as $n \to \infty$, 
$$
A^l\bigl(\tfrac{\eps^l}{A^l})^{\theta'} \lesssim |\langle \phi^{l+1},\phi\rangle| \lesssim \|\phi^{l+1}\|_{L^2_x},
$$
so by \eqref{E:wnl decoup}, 
\begin{equation} \label{E:wnl+1 small}
\limsup_{n \to \infty} \|w_n^{l+1}\|_{L^2_x} \leq A^l \sqrt{1-c\bigl(\tfrac{\eps^l}{A^l}\bigr)^{2\theta'}}.
\end{equation}
Thus by the Strichartz inequality, after iterating, we obtain \eqref{eq:error-term}. (Indeed for \eqref{eq:error-term} to fail, $\eps^l$ must stay large, but in this case \eqref{E:wnl+1 small} implies that $A^l$ must decrease to 0, a contradiction to the Strichartz inequality.)

By changing $\phi^{l+1}$ if necessary, we may assume that either $|h_n^{l+1}\xi_n^{l+1}| \to \infty$ or $\xi_n^{l+1} \equiv 0$.  Indeed, if $|h_n^{l+1}\xi_n^{l+1}| \not\to \infty$, after passing to a subsequence, $h_n^{l+1}\xi_n^{l+1} \to \xi_0 \in \R^d$, and so we may replace $\phi^{l+1}$ with $e^{ix\xi_0}\phi^{l+1}$ and $\xi_n^{l+1}$ with 0.  Similar arguments justify the assertions that for each $k < l+1$, either $\bigl(\tfrac{h_n^k}{h_n^{l+1}} + \tfrac{h_n^{l+1}}{h_n^k}\bigr) \to \infty$ or $h_n^k \equiv h_n^{l+1}$ and that either $|h_n^k\xi_n^k-h_n^{l+1}\xi_n^{l+1}| \to \infty$ or $h_n^k\xi_n^k \equiv h_n^{l+1}\xi_n^{l+1}$.  

Finally, we turn to \eqref{E:ortho}.  The crux of the argument will be the following.

\begin{lemma} \label{L:L2 ortho}
If the limit in \eqref{E:ortho} is infinite, $(B_n^k)^{-1}B_n^j \to 0$ in the weak operator topology on $\mathcal B(L^2_x)$.  Otherwise, after passing to a subsequence, there exists an $L^2_x$ isometry $B^{kj}$ such that $(B_n^k)^{-1}B_n^j \to B^{kj}$ in the strong operator topology on $\mathcal B(L^2_x)$.
\end{lemma}

\begin{proof}
Let $B_n^{kj} = (B_n^k)^{-1}B_n^j$.  It suffices to prove that if the limit in \eqref{E:ortho} is infinite, 
$$
\lim_{n \to \infty} \langle B_n^{kj}\phi, \psi \rangle =0, 
$$
for all Schwartz functions $\phi,\psi$ with compact frequency support, and that otherwise after passing to a subsequence, $B_n^{kj}\phi \to B^{kj}\phi$ in $L^2_x$ for all $\phi \in L^2_x$.  

A simple computation shows that 
$$\|B_n^{kj}\phi\|_{L^\infty_x} \leq \bigl(\tfrac{h_n^k}{h_n^j}\bigr)^{\frac d2}\|\hat\phi\|_{L^1_\xi}, \qquad \|(B_n^{kj})^{-1}\psi\|_{L^\infty_x} \leq \bigl(\tfrac{h_n^j}{h_n^k}\bigr)^{\frac d2}\|\hat\psi\|_{L^1_\xi}.$$
By H\"older's inequality, and Plancherel,
$$
|\langle B_n^{kj}\phi,\psi\rangle| \lesssim \min\{\bigl(\tfrac{h_n^k}{h_n^j}\bigr)^{\frac d2}\|\hat\phi\|_{L^1_\xi}\|\hat\psi\|_{L^\infty_\xi}, \, \bigl(\tfrac{h_n^j}{h_n^k}\bigr)^{\frac d2}\|\hat\psi\|_{L^1_\xi}\|\hat\phi\|_{L^\infty_\xi}\},
$$
and since $\phi,\psi$ are Schwartz, if $\bigl(\tfrac{h_n^k}{h_n^j} + \tfrac{h_n^j}{h_n^k}\bigr) \to \infty$, the right hand side of the above inequality tends to 0.  Thus we may henceforth assume that $h_n^j \equiv h_n^k \equiv h_n$.  

Now assume that $|h_n\xi_n^k - h_n\xi_n^j| \to \infty$.  By assumption, $\hat\phi,\hat\psi$ have compact support; say $\supp\hat\phi, \supp\hat\psi \subset\{|\xi| \leq R\}$.  Then
$$
\supp \widehat{B_n^j\phi} \subset\{|\xi-\xi_n^j| < h_n^{-1}R\}, \qquad \supp\widehat{B_n^k\psi} \subset\{|\xi-\xi_n^k| < h_n^{-1}R\}.
$$
For sufficiently large $n$, these sets are disjoint, so 
$$
\langle B_n^{kj}\phi,\psi\rangle = \langle \widehat{B_n^j\phi},\widehat{B_n^k\psi}\rangle \to 0.
$$
Thus we may assume that $\xi_n^j \equiv \xi_n^k \equiv \xi_n$.  

With these assumptions in place, we compute
$$
B_n^{kj}\phi = \omega_n^{kj}T_n^{kj}S_n^{kj}R_n^{kj}P_n^{kj}\phi,
$$
where
\begin{gather*}
\omega_n^{kj} = \exp[i(\xi_n(x_n^k-x_n^j)+(t_n^k-t_n^j)(|\xi_n|^4+\mu|\xi_n|^2))],\\
T_n^{kj}\psi(x) = \psi(x+\tfrac{x_n^k-x_n^j+(t_n^k-t_n^j)(4|\xi_n|^2\xi_n+2\mu\xi_n)}{h_n}),\\
S_n^{kj}\psi = \int \exp[i(x\xi+\tfrac{t_n^k-t_n^j}{h_n^2}(2|\xi|^2|\xi_n|^2+4(\xi\xi_n)^2+\mu|\xi|^2))]\hat\psi(\xi)\, d\xi,\\
R_n^{kj}\psi(x) = \int\exp[i(x\xi+\tfrac{4(t_n^k-t_n^j)}{h_n^3}|\xi|^2\xi\xi_n)]\hat\psi(\xi)\, d\xi,\\
P_n^{kj} \psi = \exp[\tfrac{i(t_n^k-t_n^j)}{h_n^4}\Delta^2]\psi.
\end{gather*}

After passing to a subsequence, $\omega_n^{kj} \to \omega \in S^1$, so $B_n^{jk}\phi-\omega T_n^{kj}S_n^{kj}R_n^{kj}P_n^{kj}\phi \to 0$ in $L^2_x$.  Thus it suffices to show that the conclusions of the lemma hold for $T_n^{kj}S_n^{kj}R_n^{kj}P_n^{kj}$ instead of $B_n^{kj}$.

We write
$$
T_n^{kj}S_n^{kj}R_n^{kj}P_n^{kj}\phi(x) = \int e^{i\Phi_n^{kj}(x,\xi)}\hat\phi(\xi)\, d\xi,
$$
where
\begin{align*}
\Phi_n^{kj}(x,\xi) &= \xi\cdot\left(x+\frac{x_n^k-x_n^j+(t_n^k-t_n^j)(4|\xi_n|^2\xi_n+2\mu\xi_n)}{h_n}\right)+\frac{(t_n^k-t_n^j)|\xi|^4}{h_n^4}\\& \qquad
+\frac{(t_n^k-t_n^j)(2|\xi_n|^2|\xi|^2+4(\xi_n\xi)^2+\mu|\xi|^2)}{h_n^2} + \frac{4(t_n^k-t_n^j)|\xi|^2\xi_n\xi}{h_n^3}.
\end{align*}
Since $\partial_{\xi_1}^4\Phi_n^{kj}(x,\xi) = \tfrac{24(t_n^k-t_n^j)}{h_n^4}$, by the method of stationary phase (\cite[Chapter VIII.2.2]{Stein:1993}), 
$$
|T_n^{kj}S_n^{kj}R_n^{kj}P_n^{kj}\phi(x)|\lesssim (1+\tfrac{|t_n^k-t_n^j|}{h_n^4})^{-\frac14}.
$$
Thus if $\frac{|t_n^k-t_n^j|}{h_n^4} \to \infty$, $T_n^{kj}S_n^{kj}R_n^{kj}P_n^{kj}\phi \rightharpoonup 0$, weakly in $L^2_x$.  Thus we may assume that $\frac{|t_n^k - t_n^j|}{h_n^4} \not\to\infty$.  Passing to a subsequence, $\frac{t_n^k-t_n^j}{h_n^4} \to s^{kj}$, so 
$$
\|P_n^{kj}\phi - e^{is^{kj}\Delta^2}\phi\|_{L^2_x} \to 0,
$$
for every $\phi \in L^2_x$.  Thus it suffices to show that the conclusions of the lemma hold for $T_n^{kj}S_n^{kj}R_n^{kj}$ instead of $B_n^{kj}$.  

Passing to a subsequence, either $\xi_n \equiv 0$ or $\frac{\xi_n}{|\xi_n|} \to \xi_0 \in S^{d-1}$.  The case when $\xi_n \equiv 0$ is much easier, so we assume henceforth that $\frac{\xi_n}{|\xi_n|} \to \xi_0 \in S^{d-1}$.  Passing to a further subsequence, $0 < \xi_n\cdot\xi_0\sim |\xi_n|$ for all $n$.     We write
$$
T_n^{kj}S_n^{kj}R_n^{kj} \phi(x) = \int e^{i\Phi_n^{kj}(x,\xi)}\hat\psi(x)\, d\xi,
$$
where now we set
\begin{align*}
\Phi_n^{kj}(x,\xi) &= \xi\cdot\left(x+\frac{x_n^k-x_n^j+(t_n^k-t_n^j)(4|\xi_n|^2\xi_n+2\mu \xi_n)}{h_n}\right)+\\& \qquad
+\frac{(t_n^k-t_n^j)(2|\xi_n|^2|\xi|^2+4(\xi_n\xi)^2+\mu|\xi|^2)}{h_n^2} + \frac{4(t_n^k-t_n^j)|\xi|^2\xi_n\xi}{h_n^3}.
\end{align*}
Since
$$
|(\xi_0\cdot\nabla_\xi)^3\Phi_n^{kj}(x,\xi)| = |\tfrac{24(t_n^k-t_n^j)|\xi_n|}{h_n^3} \tfrac{\xi_n\xi_0}{|\xi_n|}| \sim \tfrac{|t_n^k-t_n^j||\xi_n|}{h_n^3},
$$
if $\frac{|t_n^k-t_n^j||\xi_n|}{h_n^3} \to \infty$, by stationary phase (as above), $\|T_n^{kj}S_n^{kj}R_n^{kj}\phi\|_{L^\infty_x} \to 0$, so $T_n^{kj}S_n^{kj}R_n^{kj}\phi\rightharpoonup 0$, weakly in $L^2_x$.  Otherwise, as above, after passing to a subsequence, $R_n^{kj}\phi \to R^{kj}\phi$ in $L^2_x$, for some unitary operator $R^{kj}$.  Thus it suffices to consider $T_n^{kj}S_n^{kj}$.  

Similar arguments show that $T_n^{kj}S_n^{kj}\phi \rightharpoonup 0$, weakly in $L^2_x$ if $\frac{|t_n^k-t_n^j|(2|\xi_n|^2+\mu)}{h_n^2} \to \infty$, so we may assume that this term is bounded, and after passing to a subsequence, $S_n^{kj}\phi \to S^{kj}\phi$ in $L^2_x$.  This reduces matters to proving that the conclusions of the lemma hold for $T_n^{kj}$, and since $T_n^{kj}$ is just a translation, elementary arguments show that if $\tfrac{|x_n^k-x_n^j+(t_n^k-t_n^j)(4|\xi_n|^2\xi_n+2\mu\xi_n)|}{h_n} \to \infty$, then $T_n^{kj}\phi \rightharpoonup 0$, weakly in $L^2_x$ and if $\tfrac{x_n^k-x_n^j+(t_n^k-t_n^j)(4|\xi_n|^2\xi_n+2\mu\xi_n)}{h_n} \to y^{kj} \in \R^d$, then $T_n^{kj}\phi \to \phi(\cdot+y^{kj})$ in $L^2_x$. This completes the proof of the lemma.  
\end{proof}

Now we complete the proof of the linear profile decomposition by showing that \eqref{E:ortho} holds for all $1 \leq k < j = l+1$.  Suppose \eqref{E:ortho} failed for some $1 \leq k < j = l+1$.  Then 
\begin{align*}
0 &= \wlim_{n \to \infty} ((B_n^k)^{-1}w_n^{k-1} - \phi^k) = \wlim_{n \to \infty}(B_n^k)^{-1}w_n^k \\&= \wlim_{n \to \infty} (B_n^k)^{-1}\bigl[B_n^{k+1}\phi^{k+1}+\cdots+B_n^l\phi^l + w_n^l].
\end{align*}
We have assumed that \eqref{E:ortho} holds for all $k < j \leq l$, so by Lemma~\ref{L:L2 ortho},
$$
\wlim_{n \to \infty}(B_n^k)^{-1}\bigl[B_n^{k+1}\phi^{k+1}+\cdots+B_n^l\phi^l] = 0,
$$
which implies that 
$$
\wlim_{n \to \infty} (B_n^k)^{-1}w_n^l = 0.
$$
On the other hand, by Lemma~\ref{L:L2 ortho}, there exists a unitary operator $B^{l+1,k}$ such that after passing to a subsequence, $(B_n^{l+1})^{-1}B_n^k \to B^{l+1,k}$ in the strong operator topology on $L^2_x$; thus for any test function $\psi$,
\begin{align*}
\langle (B^{l+1,k})^{-1}\phi^{l+1},\psi\rangle &= \langle \phi^{l+1},B^{l+1,k}\psi\rangle = \lim_{n \to \infty} \langle B^{l+1}_nw_n^l,B^{l+1,k}\psi\rangle \\
&= \lim_{n \to \infty} \langle B^{l+1}_n w_n^l, (B_n^{l+1})^{-1}B_n^k \psi \rangle = \lim_{n \to \infty} \langle (B^k_n)^{-1}w_n^l,\psi\rangle = 0.
\end{align*}
Since $B^{l+1,k}$ is unitary and $\phi^{l+1} \not\equiv 0$, this is a contradiction.  Thus \eqref{E:ortho} must hold for all $1 \leq k < j \leq l+1$, and this completes the proof of Theorem~\ref{thm:linear-profile}.
\end{proof}

Next we prove Proposition~\ref{P:ST ortho}.  
\begin{proof}[Proof of Proposition~\ref{P:ST ortho}]
Since 
$$
\psi \mapsto D_\mu^{\frac d{d+2}}S_\mu(t)S_\mu(-t_n^j) g_n^j[e^{i(\cdot)h_n^j\xi_n^j}\psi]
$$ 
is a bounded linear operator from $L^2_x \to L^{\frac{2(d+2)}d}_{t,x}$, with operator norm bounded by a constant independent of $\mu$, $h_n^j$, $\xi_n^j$, $x_n^j$, $t_n^j$, by standard approximation arguments, it suffices to prove \eqref{E:ST ortho} for $\phi^j$ and $\phi^k$ lying in some dense subclass of $L^2_x$.  We will assume henceforth that they are Schwartz functions whose Fourier transforms are supported on a compact set that does not contain 0.  

Our proof will use the following pointwise upper bounds for $|D_\mu^{\frac d{d+2}} S_\mu(t) e^{i(\cdot) a}\psi(x)|$.   

\begin{lemma} \label{L:ptwise ub}
Fix $\mu \geq 0$ and let $\psi$ be a Schwartz function with compact frequency support that does not contain 0.  There exists an $L^{\frac{2(d+2)}d}_{t,x}$ function $v=v_\psi$, depending only on $\psi$, such that 
\begin{equation} \label{E:v 0}
|D_\mu^{\frac d{d+2}}S_\mu(t)\psi(x)| \leq  (1+\mu)^{\frac d{2(d+2)}} v((1+\mu) t,x),
\end{equation}
and such that if $|a| \gg \max\{\jp\xi : \xi \in \supp \hat\psi\}$, 
\begin{equation} \label{E:v a}
|D_\mu^{\frac d{d+2}}S_\mu(t)e^{i(\cdot)a}\psi(x)| \leq (2|a|^2+\mu)^{\frac d{2(d+2)}}v((2|a|^2+\mu)t,x+(4|a|^2+2\mu)at)
\end{equation}
\end{lemma}

\begin{proof}[Proof of Lemma~\ref{L:ptwise ub}]
We give the details for the second case, when $|a| \gg \max\{\jp{\xi} : \xi \in \supp \hat \psi\}$.  The case when $a=0$ is similar, but a little simpler.  

Consider the function $w_{a,\mu}(\xi) = \bigl(\frac{\sqrt{\mu+|\xi+a|^2}}{\sqrt{\mu+2|a|^2}}\bigr)^{\frac d{d+2}}$.  Then $w_{a,\mu}$ and all of its derivatives are bounded on the support of $\hat\psi$, uniformly in $a$ and $\mu$.  This follows from a simple induction argument and the fact that 
$$
c_\psi \leq |\xi+a| \leq \sqrt{\mu+|\xi+a|^2} \leq \sqrt{\mu+|a|^2}+|\xi| \leq 2\sqrt{\mu+2|a|^2},
$$
for all $\xi \in \supp \hat\psi$.  

Fix $x$ and define $y = x+t(4|a|^2a + 2\mu a)$.  Then 
\begin{align}\notag
&|D_\mu^{\frac d{d+2}}S_\mu(t) e^{i(\cdot)a}\psi(x)|
= (\mu+2|a|^2)^{\,\frac d{2(d+2)}}\left|\int_{\R^d} e^{i(x\xi+t(|\xi|^4+\mu|\xi|^2)} w_{a,\mu}(\xi-a) \hat\psi(\xi-a)\, d\xi\right|\\ \label{E:stuff phi}
&\qquad =  (\mu+2|a|^2)^{\,\frac d{2(d+2)}}\left|\int_{\R^d}e^{i(y\xi+t(|\xi|^4+4|\xi|^2\xi a+2|\xi|^2|a|^2+4(\xi a)^2 + \mu|\xi|^2)}w_{a,\mu}(\xi)\hat\psi(\xi)\, d\xi\right|.
\end{align}
Define $\Phi = \Phi_{t,y,a,\mu}$ by
$$
\Phi(\xi) = y\xi+t(|\xi|^4 + 4|\xi|^2\xi \cdot a + 2|\xi|^2|a|^2 + 4(\xi a)^2 + \mu|\xi|^2).
$$
We compute the gradient and Hessian of $\Phi$:
\begin{gather}
\label{E:grad Phi}
\nabla\Phi(\xi) = y + t(4|\xi|^2\xi + 8(\xi a)\xi + 4|\xi|^2 a + 4|a|^2\xi + 8(\xi a)a + 2\mu\xi),\\
\label{E:Hess Phi}
H_\Phi = \bigl(\Phi_{ij}(\xi)\bigr) = t\bigl(8\xi_i\xi_j + 4|\xi|^2\delta_{ij} + 8 \xi_i a_j + 8(\xi a)\delta_{ij}+ 8a_i\xi_j + 4|a|^2\delta_{ij} + 8a_i a_j + 2\mu\delta_{ij}\bigr),
\end{gather}
where $\delta_{ij}$ equals 1 if $i=j$ and 0 otherwise.  

We now prove the estimate \eqref{E:v a} by standard techniques from harmonic analysis (see Stein \cite[Chapter 8]{Stein:1993}). By H\"older's inequality and \eqref{E:stuff phi},
\begin{equation} \label{E:stuff phi lesssim 1}
|D_\mu^{\frac d{d+2}}S_\mu(t) e^{i(\cdot)a}\psi(x)| \leq (\mu+2|a|^2)^{\frac d{2(d+2)}}\|w_{a,\mu}\|_{L^\infty_\xi}\|\hat\psi\|_{L^1_\xi} \lesssim (\mu+2|a|^2)^{\frac d{2(d+2)}}, 
\end{equation}
for all $(t,x) \in \R^{1+d}$.  

On the support of $\hat\psi$, 
$$
\bigl|4|\xi|^2\xi + 8(\xi a)\xi + 4|\xi|^2 a + 4|a|^2\xi + 8(\xi a)a + 2\mu\xi\bigr| \lesssim \mu+2|a|^2,
$$
so if $(\mu+2|a|^2)|t| \ll |y|$, $|\nabla \Phi(\xi)| \gtrsim |y|$ throughout the support of $\hat\psi$.  Therefore for any $N \geq 1$, integrating by parts $N$ times in the right side of \eqref{E:stuff phi},
\begin{equation} \label{E:stuff phi x decay}
|D_\mu^{\frac d{d+2}}S_\mu(t) e^{i(\cdot)a}\psi(x)| \leq C_{\psi,N} (\mu+2|a|^2)^{\frac d{2(d+2)}}\bigl(\tfrac 1{1+|y|}\bigr)^N, 
\end{equation}
whenever $(\mu+2|a|^2)|t| \ll |y|$.  

If $(\mu+2|a|^2)|t|\gtrsim |y|$, then $\nabla \Phi$ may vanish on the support of $\hat\psi$, so we examine the Hessian of $\Phi$.  Since $|a| \gg \jp\xi$ for all $\xi \in \supp \hat \psi$, $t^{-1}H_\Phi(\xi) \geq c_\psi (\mu+2|a|^2)$ (i.e. $t^{-1}H_\Phi-c_\psi(\mu+2|a|^2)$ is a positive definite matrix).  Thus by stationary phase
\begin{equation} \label{E:stuff phi t decay}
|D_\mu^{\frac d{d+2}}S_\mu(t) e^{i(\cdot)a}\psi(x)| \leq C_\psi (\mu+2|a|^2)^{\frac d{2(d+2)}}\bigl(\tfrac 1{1+(\mu+2|a|^2)|t|}\bigr)^{-\frac d2}, 
\end{equation}
whenever $(\mu+2|a|^2)|t| \gtrsim |y|$.  

Combining \eqref{E:stuff phi lesssim 1}, \eqref{E:stuff phi x decay}, and \eqref{E:stuff phi t decay}, and recalling the definition of $y$, \eqref{E:v 0} holds with 
$$
v(s,y) = C_\psi \bigl[ \bigl(\tfrac1{1+|y|}\bigr)^{\frac d2}\chi_{|y| \gg |s|} + \bigl(\tfrac1{1+|s|}\bigr)^{\frac d2}\chi_{|y| \lesssim |s|}\bigr],
$$
and it is easy to check that $v \in L^{\frac{2(d+2)}d}_{s,y}$.  This completes the proof of Lemma~\ref{L:ptwise ub}.
\end{proof}

Now we return to estimating the quantity in \eqref{E:ST ortho}, where we recall that may assume that $\phi^k,\phi^j$ are Schwartz functions with compact frequency supports that do not contain zero.  We will use two families of $L^{\frac{2(d+2)}d}_{t,x}$ isometries:
\begin{align*}
G_n^j u(t,x) &= (h_n^j)^{-\frac{d(d+4)}{2(d+2)}} u(\tfrac{t-t_n^j}{(h_n^j)^4},\tfrac{x-x_n^j}{h_n^j}),\\
L_{a,\mu}u(t,x) &= 
\begin{cases} 
(1+\mu)^{\frac d{2(d+2)}}u((1+\mu)t,x), \quad &\text{if $a=0$}\\
(2|a|^2+\mu)^{\frac d{2(d+2)}}u((2|a|^2+\mu)t, x+(4|a|^2+2\mu)at), &\text{if $a \neq 0$.}
\end{cases}
\end{align*}

If we consider $a_n^j = h_n^j \xi_n^j$, after passing to a subsequence, either $a_n^j =0$ for all $n$, or $|a_n^j| \gg \max\{\jp{\xi} : \xi \in \supp \hat\phi^j\}$ for all $n$.  In either case, we can apply Lemma~\ref{L:ptwise ub} once we move the translation/scaling isometries across the differential operators:  
\begin{align*}
|D_\mu^{\frac d{d+2}} S_\mu(t-t_n^j)g_n^j[e^{i(\cdot)h_n^j\xi_n^j}\phi^j](x)| &= |G_n^j D_{\mu_n^j}^{\frac{d}{d+2}}S_{\mu_n^j}(t)[e^{i(\cdot)a_n^j}\phi^j](x)|\\
&\leq G_n^j L_{a_n^j,\mu_n^j} v^j(t,x),
\end{align*}
and similarly,
$$
|D_\mu^{\frac d{d+2}}S_\mu(t-t_n^k)g_n^k[e^{i(\cdot)h_n^k\xi_n^k}\phi^k](x)| \leq G_n^k L_{a_n^k,\mu_n^k}v^k(t,x),
$$
where $v^j \in L^{\frac{2(d+2)}d}_{t,x}$, $a_n^j = h_n^j\xi_n^j$, $\mu_n^j = (h_n^j)^2\mu$, and similarly with $j$ replaced by $k$.  

The proof of the proposition will thus be complete once we prove the following.

\begin{lemma} \label{L:vj vk}
If $v^j,v^k$ are any $L^{\frac{2(d+2)}d}_{t,x}$ functions and the orthogonality condition \eqref{E:ortho} holds, then 
\begin{equation} \label{E:vj vk}
\lim_{n \to \infty} \|[G_n^jL_{a_n^j,\mu_n^j}v^j][G_n^kL_{a_n^k,\mu_n^k}v^k]\|_{L^{\frac{d+2}d}_{t,x}} = 0.
\end{equation}
\end{lemma}

\begin{proof}[Proof of Lemma~\ref{L:vj vk}]
Since the $G_n^{(\cdot)}$ and $L_{a,\mu}$ operators are uniformly bounded on $L^{\frac{2(d+2)}d}_{t,x}$, it suffices to prove this for $v^j$ and $v^k$ lying in some dense subclass of $L^{\frac{2(d+2)}d}_{t,x}$.  We assume henceforth that they are compactly supported Schwartz functions; say $\supp v^j,\supp v^k \subseteq \{(t,x):\,|(t,x)| \leq R\}$.  

Passing to a subsequence, we may assume that each of the summands in \eqref{E:ortho} has a limit.  (This passage is harmless because to prove \eqref{E:vj vk}, it suffices to prove that every subsequence has a further subsequence along which the limit is zero.)  

We first consider the case when $\frac{h_n^j}{h_n^k} \to \infty$.  Using the $L^{\frac{2(d+2)}d}_{t,x}$ isometry properties and H\"older's inequality,
\begin{equation}\label{E:hnjhnk 2(d+2)/d}
\begin{aligned}
&\|[G_n^jL_{a_n^j,\mu_n^j}v^j][G_n^kL_{a_n^k,\mu_n^k}v^k]\|_{L^{\frac{d+2}d}_{t,x}}
= \|[v^j][L_{a_n^j,\mu_n^j}^{-1}(G_n^j)^{-1}G_n^k L_{a_n^k,\mu_n^k}v^k]\|_{L^{\frac{d+2}d}_{t,x}}\\
&\qquad\qquad\leq \|v^j\|_{L^{\frac{2(d+2)}d}_{t,x}(\supp(L_{a_n^j,\mu_n^j}^{-1}(G_n^j)^{-1}G_n^k L_{a_n^k,\mu_n^k}v^k))} \|L_{a_n^j,\mu_n^j}^{-1}(G_n^j)^{-1}G_n^k L_{a_n^k,\mu_n^k}v^k\|_{L^{\frac{2(d+2)}d}_{t,x}}.\\
&\qquad\qquad = \|v^j\|_{L^{\frac{2(d+2)}d}_{t,x}(\supp(L_{a_n^j,\mu_n^j}^{-1}(G_n^j)^{-1}G_n^k L_{a_n^k,\mu_n^k}v^k))} \|v^k\|_{L^{\frac{2(d+2)}d}_{t,x}}.
\end{aligned}
\end{equation}

The exact computation of $L_{a_n^j,\mu_n^j}^{-1}(G_n^j)^{-1}G_n^k L_{a_n^k,\mu_n^k}v^k$ is elementary but tedious; however it is painless to verify that 
$$
L_{a_n^j,\mu_n^j}^{-1}(G_n^j)^{-1}G_n^k L_{a_n^k,\mu_n^k}v^k(t,x) = c_n^{jk} v^k(d_n^{jk} t-s_n^{jk}, \tfrac{h_n^j}{h_n^k}x-y_n^{jk}(t)),
$$
for positive constants $c_n^{jk},d_n^{jk}$, real numbers $s_n^{jk}$, and functions $y_n^{jk}:\R \to \R^d$.  Fix $t$.  If 
$$
(t,x) \in \supp(L_{a_n^j,\mu_n^j}^{-1}(G_n^j)^{-1}G_n^k L_{a_n^k,\mu_n^k}v^k),
$$
 then $|x-\tfrac{h_n^k}{h_n^j}y_n^{jk}(t)| \leq \tfrac{h_n^k}{h_n^j}R$.  By H\"older,
$$
\|v^j(t,\cdot)\|_{L^{\frac{2(d+2)}d}_x(\{|x-y_n^j(t)| \leq \frac{h_n^k}{h_n^j}R\})}^{\frac{2(d+2)}d}
\lesssim 
\bigl(\tfrac{h_n^k}{h_n^j}R\bigr)^d \|v^j\|_{L^\infty_{t,x}}.
$$
Integrating the above estimate with respect to $t$ and recalling that $\tfrac{h_n^j}{h_n^k}\to \infty$,
$$
\|v^j\|_{L^{\frac{2(d+2)}d}_{t,x}(\supp(L_{a_n^j,\mu_n^j}^{-1}(G_n^j)^{-1}G_n^k L_{a_n^k,\mu_n^k}v^k))}^{\frac{2(d+2)}d} \lesssim R(\tfrac{h_n^k}{h_n^j}R\bigr)^d \to 0.
$$
By \eqref{E:hnjhnk 2(d+2)/d}, this implies \eqref{E:vj vk}.  By a similar argument, \eqref{E:vj vk} also holds if $\tfrac{h_n^k}{h_n^j} \to \infty$.  

Henceforth, we may assume that $h_n^j \equiv h_n^k\equiv h_n$, so $\mu_n^j \equiv \mu_n^k \equiv \mu_n$ as well.  Using a change of variables, we may now remove the dilations from the $G_n^{(\cdot)}$:  
\begin{equation} \label{E:remove h}
\|[G_n^jL_{a_n^j,\mu_n^j}v^j][G_n^kL_{a_n^k,\mu_n^k}v^k]\|_{L^{\frac{d+2}d}_{t,x}} = \|[L_{a_n^j,\mu_n^j}v^j][G_n^{jk}L_{a_n^k,\mu_n^k}v^k]\|_{L^{\frac{d+2}d}_{t,x}},
\end{equation}
where $G_n^{jk}v(t,x) = v(t-\tfrac{t_n^k-t_n^j}{h_n^4},x-\tfrac{x_n^k-x_n^j}{h_n})$.

We now break into three cases:  $a_n^j \equiv a_n^k \equiv 0$; $a_n^j \equiv 0$ and $|a_n^k|\to \infty$; and $|a_n^j|,|a_n^k| \to \infty$.  

We deal with the easiest of the three cases, $a_n^j \equiv a_n^k \equiv 0$, first.  In this case,
\begin{align*}
L_{a_n^j,\mu_n}v^j(t,x) &= (1+\mu_n)^{\frac d{2(d+2)}}v^j((1+\mu_n)t,x),\\
G_n^{jk}L_{a_n^k,\mu_n}v^k(t,x) &= (1+\mu_n)^{\frac d{2(d+2)}}v^k((1+\mu_n)[t-\tfrac{t_n^k-t_n^j}{h_n^4}],x-\tfrac{x_n^k-x_n^j}{h_n}).
\end{align*}
By our assumptions on the various parameters, including the assumption that the orthogonality condition \eqref{E:ortho} holds, either $\tfrac{t_n^k-t_n^j}{h_n^4} \to \infty$ or $\tfrac{x_n^k-x_n^j}{h_n} \to \infty$.  In either case, $\supp(L_{a_n^j,\mu_n} v^j)\cap \supp(G_n^{jk}L_{a_n^k,\mu_n}v^k)$ is empty for sufficiently large $n$, so the right hand side of \eqref{E:remove h} is eventually zero.  By the identity \eqref{E:remove h}, this establishes \eqref{E:vj vk}.  
 
We turn now to the case when $a_n^j \equiv 0$, $|a_n^k| \to \infty$.  (By symmetry, this argument also covers the case when the roles of $j$ and $k$ are reversed.)  Arguing similarly to \eqref{E:hnjhnk 2(d+2)/d}, 
\begin{equation} \label{E:parallelogram}
\|[L_{a_n^j,\mu_n} v^j][G_n^{jk}L_{a_n^k,\mu_n}v^k]\|_{L^{\frac{d+2}d}_{t,x}} \lesssim \|v_j\|_{L^{\frac{d+2}d}_{t,x}(\supp L_{a_n^j,\mu_n}^{-1}G_n^{jk}L_{a_n^k,\mu_n} v^k)}.
\end{equation}
We compute:
\begin{align*}
L_{a_n^j,\mu_n}^{-1}G_n^{jk}L_{a_n^k,\mu_n}v^k(t,x) &= \bigl(\tfrac{2|a_n^k|^2+\mu_n}{1+\mu_n}\bigr)^{\frac d{2(d+2)}} v^k(\tfrac{2|a_n^k|^2+\mu_n}{1+\mu_n}(t-\tfrac{(t_n^k-t_n^j)(1+\mu_n)}{h_n^4}), \\ 
&\qquad\qquad  x-\tfrac{x_n^k-x_n^j}{h_n}+\tfrac{4|a_n^k|^2+2\mu_n}{1+\mu_n}a_n^k(t-\tfrac{(t_n^k-t_n^j)(1+\mu_n)}{h_n^4}))\\
&= \bigl(\tfrac{2|a_n^k|^2+\mu_n}{1+\mu_n}\bigr)^{\frac d{2(d+2)}} v^k(\tfrac{2|a_n^k|^2+\mu_n}{1+\mu_n}(t-s_n^{jk}), x+\tfrac{4|a_n^k|^2+2\mu_n}{1+\mu_n} a_n^k t - y_n^{jk}),
\end{align*}
where 
\begin{gather*}
s_n^{jk} = \tfrac{(t_n^k-t_n^j)(1+\mu_n)}{h_n^4}\\
y_n^{jk} = \tfrac{x_n^k-x_n^j}{h_n} + \tfrac{4|a_n^k|^2+2\mu_n}{1+\mu_n}a_n^k s_n^{jk}.
\end{gather*}

Fix $x$.  If $(t,x) \in (\supp v^j) \cap (\supp L_{a_n^j,\mu_n}^{-1}G_n^{jk}L_{a_n^k\mu_n} v^k)$, it satisfies $|x| \leq R$ and $|x+\tfrac{4|a_n^k|^2+2\mu_n}{1+\mu_n} a_n^k t - y_n^{jk}| \leq R$.  Therefore 
$$
|\tfrac{4|a_n^k|^2+2\mu_n}{1+\mu_n} a_n^k t - y_n^{jk}| \leq 2R,
$$
so recalling that $|a_n^k| \gtrsim 1$ for all $n$,
$$
|t-r_n^{jk}| \leq \tfrac{2R(1+\mu_n)}{(4|a_n^k|^2+2\mu_n)|a_n^k|} \lesssim \tfrac{R}{|a_n^k|},
$$
where 
$$
r_n^{jk} = \tfrac{(1+\mu_n)y_n^{jk}\cdot a_n^k}{(4|a_n^k|^2+2\mu_n)|a_n^k|^2}.
$$
Integrating in $t$,
$$
\|v^j(\cdot,x)\|_{L^{\frac{2(d+2)}d}_t(\{|t+r_n^{jk}| \lesssim \frac{R}{|a_n^k|}\})}^{\frac{2(d+2)}d} \lesssim \bigl(\|v^j\|_{L^\infty_{t,x}}\bigr) \tfrac{R}{|a_n^k|}.
$$
Integrating this in $x$,
$$
\|v^j\|_{L^{\frac{d+2}d}_{t,x}(\supp L_{a_n^j,\mu_n}^{-1}G_n^{jk}L_{a_n^k\mu_n} v^k)}^{\frac{2(d+2)}d} \lesssim \|v^j\|_{L^\infty_{t,x}} R^d \tfrac{R}{|a_n^k|} \to 0.
$$
This implies \eqref{E:vj vk}, and completes the case when $a_n^j \equiv 0$, $|a_n^k| \to \infty$.  

Finally we turn to the case when $|a_n^j|,|a_n^k| \to \infty$.  We compute
$$
L_{a_n^j,\mu_n}^{-1}G_n^{jk}L_{a_n^k,\mu_n} v^k(t,x) = (c_n^{jk})^{\frac d{2(d+2)}}v^k(c_n^{jk}(t-s_n^{jk}),x-y_n^{jk}-b_n^{jk}t),
$$
where
\begin{align*}
c_n^{jk} &= \tfrac{2|a_n^k|^2+\mu_n}{2|a_n^j|^2+\mu_n},\\
s_n^{jk} &= \tfrac{(t_n^k-t_n^j)(2|a_n^j|^2+\mu_n)}{h_n^4},\\
y_n^{jk} &= \tfrac{x_n^k-x_n^j}{h_n}+a_n^k\tfrac{(4|a_n^k|^2+2\mu_n)(t_n^k-t_n^j)}{h_n^4},\\
b_n^{jk} &= \tfrac{(2|a_n^j|^2+\mu_n)a_n^j-(2|a_n^k|^2+\mu_n)a_n^k}{2|a_n^j|^2+\mu_n}.
\end{align*}

Suppose that $|a_n^j-a_n^k| \to \infty$.  We claim that $|b_n^{jk}| \to \infty$.  Indeed, 
\begin{align*}
|b_n^{jk}| &\geq b_n^{jk} \cdot \frac{a_n^k-a_n^j}{|a_n^j-a_n^k|} \\&
= \frac{5[a_n^j\cdot(a_n^k-a_n^j)]^2+6|a_n^k-a_n^j|^2a_n^j \cdot(a_n^k-a_n^j) + 2|a_n^k-a_n^j|^4 + |a_n^k-a_n^j|^2(|a_n^j|^2+\mu_n)}{(2|a_n^j|^2+\mu_n)|a_n^k-a_n^j|} \\
&\geq \frac{|a_n^k-a_n^j|^2(|a_n^j|^2+\mu_n)}{(2|a_n^j|^2+\mu_n)|a_n^k-a_n^j|} \geq \tfrac12|a_n^k-a_n^j|,
\end{align*}
where for the second inequality, we used Cauchy--Schwartz and the elementary inequality
$$
5(ab)^2+2b^4 \geq \tfrac92(ab)^2+2b^4 \geq 6|ab^3|,
$$
for all $a,b \in \R$.  
Since $|b_n^{jk}| \to \infty$, arguing exactly as we did in the case $a_n^j \equiv 0$, $|a_n^k| \to \infty$, we can establish \eqref{E:vj vk}.  

Henceforth, we may assume that $a_n^j \equiv a_n^k$.  Thus $c_n^{jk} \equiv 1$ and $b_n^{jk} \equiv 0$.  If $|s_n^{jk}| \to \infty$ or $|y_n^{jk}| \to \infty$, the (compact) supports of $v^j$ and $L_{a_n^j\mu_n}^{-1}G_n^{jk}L_{a_n^k\mu_n} v^k$ are disjoint for sufficiently large $n$, so \eqref{E:vj vk} holds.  Since $|s_n^{jk}| \geq \tfrac{|t_n^j-t_n^k|}{h_n^4}$, this completes the proof in the case $|a_n^j|,|a_n^k| \to \infty$.
\end{proof}
Finally, the proposition is proved.  
\end{proof}


\section{Application:  Dichotomy result on the existence of extremizers} \label{S:extreme}


As an application of the profile decomposition in Theorem \ref{thm:linear-profile}, we establish lower bounds for the operator norms and a dichotomy result on existence of extremizers. Similar results have previously appeared in Jiang, Pausader and Shao \cite{Jiang-Pausader-Shao:2010:4th-NLS}. We begin by defining 
\begin{equation}\label{E:Amu}
A_\mu := \sup_{\substack{\phi\in L^2_x, \\ \phi\neq 0 }}\dfrac {\|D_\mu^{\frac{d}{d+2}} S_\mu(t) \phi\|_{L^{\frac{2(d+2)}d}_{t,x}(\mathbb{R}\times \mathbb{R}^d)}}{\|\phi\|_{L^2_x}}, \qquad \mu \geq 0,
\end{equation}
and
\begin{equation}\label{E:B}
B := \sup_{\substack{\psi\in L^2_x, \\ \psi\neq 0 }}\dfrac {\|e^{it\Delta}\psi \|_{L^{\frac{d}{d+2}}_{t,x}(\mathbb{R}\times \mathbb{R}^d)}}{\|\psi\|_{L^2_x}}.
\end{equation}
These are finite by the Strichartz inequalities for the fourth order Schr\"odinger and Schr\"odinger equations.

We say that a function $\phi$ is an extremizer for $A_\mu$ (resp.\ $B$) if $\|\phi\|_{L^2_x} \neq 0$ and $\phi$ maximizes the ratio in \eqref{E:Amu} (resp.\ \eqref{E:B}).  A sequence $\{\phi_n\}$ is an extremizing sequence for $A_\mu$ if 
$$
A_\mu = \lim_{n \to \infty} \frac{\|D_\mu^{\frac{d}{d+2}} S_\mu(t)\phi_n\|_{L^{\frac{2(d+2)}d}_{t,x}}}{\|\phi_n\|_{L^2_x}};
$$
$\{\phi_n\}$ is $L^2_x$-normalized if $\|\phi_n\|_{L^2_x} = 1$ for all $n$.  

Extremizers are known to exist for $B$ (\cite{Foschi:2007:maxi-strichartz-2d, Hundertmark-Zharnitsky:2006:maximizers-Strichartz-low-dimensions} for $d=1,2$, \cite{Shao:2009} for $d \geq 3$).  In dimensions 1 and 2, it is known in addition that the extremizers are only Gaussian functions, modulo symmetries of the Schr\"odinger equation \cite{Foschi:2007:maxi-strichartz-2d, Hundertmark-Zharnitsky:2006:maximizers-Strichartz-low-dimensions}.

\begin{theorem}\label{T:dichotomy 0}
The operator norms $A_0$ and $B$ satisfy:  
\begin{equation} \label{E:A0 B}
A_0 \geq 3^{-\frac1{2(d+2)}}2^{-\frac{d}{2(d+2)}}B.
\end{equation}
If the inequality is strict in \eqref{E:A0 B}, then extremizers exist for $A_0$.  If extremizers do not exist and $\{\phi_n\}$ is an $L^2_x$-normalized extremizing sequence for $A_0$, there exist a sequence of parameters $(h_n,\xi_n,x_n,t_n)$ with $|h_n\xi_n| \to \infty$ and an extremizer $\psi$ for $B$ such that after passing to a subsequence,
\begin{equation} \label{E:to B extreme 0}
\lim_{n \to \infty} \|\phi_n - S_0(t_n)g_n(e^{i(\cdot)h_n\xi_n}\psi \circ \ell_{h_n\xi_n}^{-1})\|_{L^2_x} = 0,
\end{equation}
where for $a \in \R^d$, $\ell_a$ denotes the transformation
\begin{equation} \label{E:ell a}
\ell_a(\xi) = \sqrt{6} \, \rm{proj}_a(\xi) + \sqrt{2}(\xi-\rm{proj}_a(\xi)),
\end{equation}
where $\rm{proj}_a(\xi) := (\xi\cdot \frac {a}{|a|})\frac {a}{|a|}$. Conversely, if equality holds in \eqref{E:A0 B}, any sequence $\{\phi_n\}$ satisfying \eqref{E:to B extreme 0} for a sequence of parameters $(h_n,\xi_n,x_n,t_n)_{n \geq 1}$ with $|h_n\xi_n| \to \infty$ and an extremizer $\psi$ for $B$ is an extremizing sequence for $A_0$.
\end{theorem}

If $\mu > 0$, then by scaling, $A_\mu=A_1$; scaling also gives a natural correspondence between extremizing sequences (and, if they exist, extremizers) for $A_\mu$ and those for $A_1$.  Thus we only state the dichotomy result in the case $\mu=1$.  

\begin{theorem}\label{T:dichotomy 1}
The operator norms satisfy $A_1 \geq \max\{A_0,B\}$, and if this inequality is strict, extremizers exist for $A_1$.  If extremizers do not exist and $\{\phi_n\}$ is an $L^2_x$-normalized extremizing sequence for $A_1$, then there exist a sequence of parameters $(h_n,\xi_n,x_n,t_n)$ and a function $\phi \in L^2_x$ such that after passing to a subsequence,
\begin{equation} \label{E:to B extreme 1}
\lim_{n \to \infty} \|\phi_n - S_1(t_n) g_n e^{i(\cdot)h_n\xi_n}\psi\|_{L^2_x} = 0.
\end{equation}
Moreover, in this case, one of the following occurs:  either $A_1=B$, $h_n \to \infty$, $|\xi_n| \to 0$, and $\psi$ is an extremizer for $B$, or $A_1=A_0$, $h_n \to 0$, $\xi_n \equiv 0$, and $\psi$ is an extremizer for $A_0$.  
\end{theorem}

The analogue of the final conclusion of Theorem~\ref{T:dichotomy 0} for $A_1$ is the following.  If $A_1=B$, $h_n \to \infty$, $|\xi_n| \to 0$, $\psi$ is an extremizer for $B$, and $\phi_n$ satisfies \eqref{E:to B extreme 1}, then $\phi_n$ is an extremizing sequence for $A_1$.  If $A_1=A_0$, then $A_0 \geq B$, so the inequality is strict in \eqref{E:A0 B}, which implies that extremizers exist for $A_0$.  Furthermore, in this case, if $h_n \to 0$, $\xi_n \equiv 0$, $\psi$ is an extremizer for $A_0$, and $\phi_n$ satisfies \eqref{E:to B extreme 1}, then $\phi_n$ is an extremizing sequence for $A_1$.  

The proofs of these theorems will rely on the linear profile decomposition and the following lemmas.

\begin{lemma} \label{L:limit S0}  
Let $\phi \in L^2_x(\R^d)$.  If $(a_n)$ is a sequence in $\R^d$ with $|a_n| \to \infty$, then 
\begin{equation} \label{E:limit S0}
\begin{aligned}
&\lim_{n \to \infty}\||\nabla|^{\frac d{d+2}} e^{it\Delta^2} (e^{i(\cdot)a_n}\phi) \\&\qquad
- |a_n|^{\frac d{d+2}} e^{i(x \cdot a_n + t|a_n|^4)} e^{-i|a_n|^2t \Delta} (\phi \circ \ell_{a_n})(\ell_{a_n}^{-1}(x+4t|a_n|^2a_n))\|_{L^{\frac{2(d+2)}d}_{t,x}}  = 0,
\end{aligned}
\end{equation}
with $\ell_a$ as in \eqref{E:ell a}.
\end{lemma}

\begin{lemma} \label{L:limit S1}
Let $\phi \in L^2_x(\R^d)$.  Let $(h_n,\xi_n)$ be a sequence in $(0,\infty) \times \R^d$.  Define $g_n\phi(x) = h_n^{-\frac d2}\phi(\frac x{h_n})$.  If $h_n \to 0$ and $\xi_n \equiv 0$,
\begin{equation} \label{E:S1 to S0}
\lim_{n \to \infty} \|\jpn^{\frac d{d+2}} S_1(t) g_n e^{i(\cdot) h_n\xi_n} \phi - |\nabla|^{\frac d{d+2}}e^{it\Delta^2} g_n \phi\|_{L^{\frac{2(d+2)}d}_{t,x}} = 0.
\end{equation}
If $h_n \to \infty$ and $\xi_n \equiv 0$,
\begin{equation} \label{E:S1 to Sch 0}
\lim_{n \to \infty} \|\jpn^{\frac d{d+2}} S_1(t) g_n e^{i(\cdot)h_n\xi_n}\phi - e^{-it\Delta} g_n \phi\|_{L^{\frac{2(d+2)}d}_{t,x}} = 0.
\end{equation}
If $|\xi_n| \lesssim 1$ for all $n$ and $|h_n\xi_n| \to \infty$,
\begin{equation}\label{E:S1 to Sch 1}
\begin{aligned}
&\lim_{n \to \infty} \|\jpn^{\frac d{d+2}} S_1(t) g_n e^{i(\cdot)h_n\xi_n}\phi  \\
&\quad - \jp{\xi_n}^{\frac d{d+2}} e^{i(x\xi_n + t|\xi_n|^4 + t|\xi_n|^2)}e^{-it\Delta}[g_n(\phi\circ\tilde\ell_n)](\tilde\ell_n^{-1}(x+4t|\xi_n|^2\xi_n+2t\xi_n))\|_{L^{\frac{2(d+2)}d}_{t,x}} = 0,
\end{aligned}
\end{equation}
where $\tilde\ell_n(\xi) = \sqrt{6|\xi_n|^2+1}\, (\rm{proj}_{\xi_n}\xi) + \sqrt{2|\xi_n|^2+1}\, (\xi-\rm{proj}_{\xi_n}\xi)$.  If $|\xi_n| \to \infty$ and $|h_n\xi_n| \to \infty$,
\begin{equation} \label{E:S1 to Sch infty}
\begin{aligned}
&\lim_{n \to \infty} \|\jpn^{\frac d{d+2}} S_1(t) g_n e^{i(\cdot)h_n\xi_n}\phi \\ 
& - |\xi_n|^{\frac d{d+2}}e^{i(x\xi_n + t|\xi_n|^4 + t|\xi_n|^2)}e^{-i\frac{|\xi_n|^2}{h_n^2}t\Delta}[g_n(\phi\circ \ell_{n})](\ell_{n}^{-1}(x+4t|\xi_n|^2\xi_n+2t\xi_n))\|_{L^{\frac{2(d+2)}d}_{t,x}} = 0,
\end{aligned}
\end{equation}
with $\ell_n=\ell_{\xi_n}$, using the notation from \eqref{E:ell a}.
\end{lemma}

The lemmas will be proved at the end of this section.  Before proceeding to their proofs, we show how Theorem~\ref{T:dichotomy 0} can be proved from Lemma~\ref{L:limit S0} and indicate how to adapt this proof for Theorem~\ref{T:dichotomy 1}.

\begin{proof}[Proof of Theorem~\ref{T:dichotomy 0}]
Let $(h_n,\xi_n,x_n,t_n)$ be a sequence of parameters with $|h_n \xi_n| \to \infty$ and let $\psi$ be an extremizer for $B$.  Assume that the sequence $\{\phi_n\} \subset L^2_x$ satisfies
$$
\lim_{n \to \infty} \|\phi_n - e^{it_n\Delta^2} g_n(e^{i(\cdot)h_n \xi_n} \psi \circ \ell_{h_n \xi_n}^{-1})\|_{L^2_x} = 0.
$$
By the Strichartz inequality \eqref{E:Strichartz}, changes of variables, Lemma~\ref{L:limit S1}, and the assumption that $\psi$ is an extremizer, 
\begin{equation} \label{E:A0 B computation}
\begin{aligned}
A_0 
&\geq \lim_{n \to \infty}\frac{\||\nabla|^{\frac{d}{d+2}}e^{it\Delta^2}\phi_n\|_{L^{\frac{2(d+2)}d}_{t,x}}}{\|\phi_n\|_{L^2_x}} 
= \lim_{n \to \infty} \frac{\||\nabla|^{\frac{d}{d+2}} e^{i(t+t_n)\Delta^2} g_n(e^{i(\cdot)h_n\xi_n}(\psi \circ \ell_{h_n\xi_n}^{-1}))\|_{L^{\frac{2(d+2)}d}_{t,x}}}{ \|e^{it_n\Delta^2} g_n(e^{i(\cdot)h_n \xi_n} \psi \circ \ell_{h_n \xi_n}^{-1})\|_{L^2_x}}  \\
&= 3^{-\frac 14}2^{-\frac d4} \|\psi\|_{L^2_x}^{-1}\lim_{n \to \infty} \||\nabla|^{\frac{d}{d+2}} e^{it\Delta^2} e^{i(\cdot)h_n \xi_n} (\psi \circ \ell_{h_n \xi_n}^{-1})\|_{L^{\frac{2(d+2)}d}_{t,x}} \\
&= 3^{-\frac14}2^{-\frac d4} \|\psi\|_{L^2_x}^{-1}\lim_{n \to \infty} \|(|h_n \xi_n|^{\frac{d}{d+2}} e^{-i|h_n\xi_n|^2t\Delta}\psi) \circ \ell_{h_n\xi_n}^{-1}\|_{L^{\frac{2(d+2)}d}_{t,x}} \\
& = 3^{-\frac1{2(d+2)}}2^{-\frac{d}{2(d+2)}} \|\psi\|_{L^2_x}^{-1}\|e^{-it\Delta}\psi\|_{L^{\frac{2(d+2)}d}_{t,x}} = 3^{-\frac1{2(d+2)}}2^{-\frac{d}{2(d+2)}}B.
\end{aligned}
\end{equation}
This verifies \eqref{E:A0 B}.  Conversely, if equality holds in \eqref{E:A0 B}, then it holds everywhere in the computation above, establishing the final conclusion of the theorem.  

In the other direction, let $\{\phi_n\}$ be an $L^2_x$-normalized extremizing sequence for $A_0$.  By Theorem~\ref{thm:linear-profile}, there exist sequences $\{\phi^j\}_{j \geq 1}$, $\{w_n^j\}_{j \geq 1, n \geq 1}$, and parameters $(h_n^j,\xi_n^j,t_n^j)_{j \geq 1, n \geq 1}$ such that for each $j$, $|h_n^j \xi_n^j| \to \infty$ or $h_n^j \xi_n^j \equiv 0$ and such that after passing to a subsequence,
$$
\phi_n = \sum_{j=1}^l e^{it_n^j \Delta^2} g_n^j(e^{i(\cdot)h_n^j \xi_n^j} \phi^j) + w_n^l,
$$
where \eqref{eq:error-term}, \eqref{E:L2 decoup}, and \eqref{E:Lp decoup} hold.

Therefore,
\begin{align*}
A_0^{\frac{2(d+2)}d} &= \lim_{n \to \infty} \||\nabla|^{\frac{d}{d+2}} e^{it\Delta^2} \phi_n\|_{L^{\frac{2(d+2)}d}_{t,x}}^{\frac{2(d+2)}d} \\
&\leq \limsup_{l \to \infty} \limsup_{n \to \infty} \sum_{1 \leq j \leq l} \||\nabla|^{\frac{d}{d+2}} e^{i(t+t_n^j)\Delta^2}g_n^j(e^{i(\cdot) h_n^j \xi_n^j} \phi^j)\|_{L^{\frac{2(d+2)}d}_{t,x}}^{\frac{2(d+2)}d}\\
&\leq \limsup_{l \to \infty} \limsup_{n \to \infty} \sum_{1 \leq j \leq l} A_0^{\frac{2(d+2)}d} \|e^{i(\cdot)h_n^j \xi_n^j}\phi^j\|_{L^2_x}^{\frac{2(d+2)}d}
= \sum_j A_0^{\frac{2(d+2)}d} \|\phi^j\|_{L^2_x}^{\frac{2(d+2)}d}\\
&\leq A_0^{\frac{2(d+2)}d} \bigl(\sup_j \|\phi^j\|_{L^2_x}^{\frac4d} \bigr)\sum_j \|\phi^j\|_{L^2_x}^2.
\end{align*}
By \eqref{E:L2 decoup}, the right hand side is strictly less than the left hand side (a contradiction) unless there exists $j$ such that $\|\phi^j\|_{L^2_x}=1$.  In this case, there is only one profile and the error terms tend to zero in $L^2_x$:
\begin{equation} \label{E:phi_n mod symm}
\phi_n = e^{it_n\Delta^2} g_n(e^{i(\cdot)h_n \xi_n}\phi) + w_n, \qquad \lim_{n \to \infty}\|w_n\|_{L^2_x} = 0.
\end{equation}

If $h_n\xi_n \equiv 0$, since 
$$
\|e^{it_n\Delta^2} g_n(\phi)\|_{L^2_x} = \|\phi\|_{L^2_x}, \qquad \||\nabla|^{\frac{d}{d+2}}e^{i(t+t_n)\Delta^2} g_n(\phi)\|_{L^{\frac{2(d+2)}d}_{t,x}} = \||\nabla|^{\frac{d}{d+2}}e^{it\Delta^2} \phi\|_{L^{\frac{2(d+2)}d}_{t,x}},
$$
$\phi$ is an extremizer for $A_0$.  Thus if $A_0$ does not have an extremizer, every $L^2_x$-normalized extremizing sequence must satisfy (after passing to a subsequence)
\begin{equation}\label{E:wrong limit}
\|\phi_n - e^{it_n\Delta^2} g_n(e^{i(\cdot)h_n \xi_n}\phi)\|_{L^2_x} \to 0,
\end{equation}
for some function $\phi \in L^2_x$ and parameters $(h_n,\xi_n,t_n,x_n)$ with $|h_n\xi_n| \to \infty$.  By the essentially the same computation as \eqref{E:A0 B computation}, this implies that $A_0 \leq 3^{-\frac1{2(d+2)}}2^{-\frac{d}{2(d+2)}}B$, and hence that equality holds in \eqref{E:A0 B}.

Thus it remains to show that if $\{\phi_n\}$ is an $L^2_x$-normalized extremizing sequence for $A_0$, \eqref{E:A0 B} holds with equality, and \eqref{E:wrong limit} holds for some $\phi$ and $(h_n,t_n,x_n,\xi_n)$ with $|h_n\xi_n| \to \infty$, then \eqref{E:to B extreme 0} holds with $\psi$ an extremizer for $B$.  

Passing to a further subsequence, there exists $\omega \in S^{d-1}$ such that $\frac{h_n\xi_n}{|h_n\xi_n|} \to \omega$.  Let $\psi = \phi \circ \ell_\omega$.  For $a \neq 0$, $\ell_a$ depends only on $\frac a{|a|}$, so $\phi - \psi \circ \ell_{h_n\xi_n}^{-1} \to 0$ in $L^2_x$.  Therefore
$$
\|e^{it_n\Delta^2}g_n(e^{i(\cdot)h_n\xi_n}\phi) - e^{it_n\Delta^2}g_n(e^{i(\cdot)h_n\xi_n}(\psi \circ \ell_{h_n\xi_n}^{-1}))\|_{L^2_x} \to 0,
$$
which implies by \eqref{E:phi_n mod symm} that
$$
\|\phi_n - e^{it_n\Delta^2}g_n(e^{i(\cdot)h_n\xi_n}(\psi \circ \ell_{h_n\xi_n}^{-1}))\|_{L^2_x} \to 0.
$$
That $\psi$ is an extremizer for $B$ follows from the same computations as in \eqref{E:A0 B computation}.  This completes the proof of Theorem~\ref{T:dichotomy 0}.
\end{proof}

\begin{proof}[Adapting the argument for Theorem~\ref{T:dichotomy 1}]
There are two relatively minor differences in the proof of Theorem~\ref{T:dichotomy 1}.  First, $A_1$ must be compared to two operator norms, $A_0$ and $B$.  To obtain the estimate $A_1 \geq B$, we simply take an extremizer $\psi$ for $B$ and use \eqref{E:S1 to Sch 0}, arguing similarly to \eqref{E:A0 B computation}.  Given $\eps>0$, we can show that $A_1 \geq (1-\eps)A_0$ by selecting an $L^2_x$-normalized function $\psi$ satisfying $\||\nabla|^{\frac{d}{d+2}}e^{it\Delta^2}\psi\| \geq (1-\eps)A_0$ and using \eqref{E:S1 to S0}; letting $\eps \to 0$, we see that $A_1 \geq A_0$.  

Second, we must rule out the case in which \eqref{E:to B extreme 1} holds for some $\psi \in L^2_x$ and some sequence of parameters $(h_n,\xi_n,t_n,x_n)$ with $\xi_n \not\to 0$.  Passing to a subsequence and using the fact that spacetime translations do not affect any of the relevant operator norms, it suffices to consider the cases when $(t_n,x_n) \equiv (0,0)$, $|h_n\xi_n| \to \infty$, and either $\xi_n \to \xi_0 \neq 0$ or $|\xi_n| \to \infty$.  If $|\xi_n| \to \infty$, we apply \eqref{E:S1 to Sch infty} and compute 
\begin{align*}
&\||\xi_n|^{\frac d{d+2}}e^{i(x\xi_n + th_n^2|\xi_n|^4 + t|\xi_n|^2)}e^{-i|\xi_n|^2t\Delta}[g_n(\phi\circ \ell_{n})](\ell_{n}^{-1}(x+th_n^2|\xi_n|^2\xi_n+2t\xi_n))\|_{L^{\frac{2(d+2)}d}_{t,x}} \\
&\qquad
\leq 3^{-\frac1{2(d+2)}}2^{-\frac{d}{2(d+2)}}B\|\phi\|_{L^2_x} < B, 
\end{align*}
provided $\xi_n \neq 0$, $\ell_n = \ell_{\xi_n}$ is as in \eqref{E:ell a}, and $\|\phi\|_{L^2_x} \leq 1$.  If $\xi_n \to \xi_0 \neq 0$ and $\|\phi\|_{L^2_x} \leq 1$, we use \eqref{E:S1 to Sch 1} and compute
\begin{align*}
&\limsup_{n \to \infty} \|\jp{\xi_n}^{\frac d{d+2}} e^{i(x\xi_n + th_n^2|\xi_n|^4 + t|\xi_n|^2)}e^{-it\Delta}[g_n(\phi\circ\tilde\ell_n)](\tilde\ell_n^{-1}(x+4th_n^2|\xi_n|^2\xi_n+2t\xi_n))\|_{L^{\frac{2(d+2)}d}_{t,x}} \\
&\qquad \leq \limsup_{n \to \infty} (|\xi_n|^2+1)^{\frac d{d+2}}(6|\xi_n|^2+1)^{-\frac 1{d+2}}(2|\xi_n|^2+1)^{-\frac{d-1}{d+2}}B\|\phi\|_{L^2_x} \\
&\qquad = (|\xi_0|^2+1)^{\frac d{d+2}}(6|\xi_0|^2+1)^{-\frac 1{d+2}}(2|\xi_0|^2+1)^{-\frac{d-1}{d+2}}B\|\phi\|_{L^2_x} < B,
\end{align*}
so this case can be ruled out as well.  
\end{proof}

Finally, we prove the lemmas.

\begin{proof}[Proof of Lemmas~\ref{L:limit S0} and~\ref{L:limit S1}]
We begin by observing that by the change of variables formula and the Strichartz inequality for Schr\"odinger,
\begin{align*}
&\||a_n|^{\frac d{d+2}}e^{i(x\cdot a_n + t|a_n|^4 + t\mu|a_n|^2)} e^{-i|a_n|^2 t\Delta}(\phi \circ \ell_{a_n})(\ell_{a_n}^{-1}(x+4t|a_n|^2a_n))\|_{L^{\frac{2(d+2)}d}_{t,x}}\\
&\qquad = 3^{\frac{d}{4(d+2)}}2^{\frac{d^2}{4(d+2)}}\|e^{-it\Delta}\phi \circ \ell_{a_n}\|_{L^{\frac{2(d+2)}d}_{t,x}} \\
& \qquad  \leq  3^{\frac{d}{4(d+2)}}2^{\frac{d^2}{4(d+2)}} B \|\phi \circ \ell_{a_n}\|_{L^2_x} = 3^{-\frac1{2(d+2)}}2^{-\frac{d}{2(d+2)}}B\|\phi\|_{L^2_x}.
\end{align*}
Similar computations give:
\begin{align*}
&\||\nabla|^{\frac d{d+2}}e^{it\Delta^2} g_n \phi\|_{L^{\frac{2(d+2)}d}_{t,x}} 
\leq A_0 \|\phi\|_{L^2_x}, \qquad 
\|e^{-it\Delta} g_n \phi\|_{L^{\frac{2(d+2)}d}_{t,x}}
\leq B\|\phi\|_{L^2_x},\\
&\|\jp{\xi_n}^{\frac d{d+2}} e^{i(x\xi_n + th_n^2|\xi_n|^4 + t|\xi_n|^2)}e^{-it\Delta}[g_n(\phi\circ\tilde\ell_n)](\tilde\ell_n^{-1}(x+4th_n^2|\xi_n|^2\xi_n+2t\xi_n))\|_{L^{\frac{2(d+2)}d}_{t,x}}\\
&\qquad \leq (|\xi_n|^2+1)^{\frac d{d+2}}(6|\xi_n|^2+1)^{-\frac 1{d+2}}(2|\xi_n|^2+1)^{-\frac{d-1}{d+1}} B\|\phi\|_{L^2_x},\\
&\||\xi_n|^{\frac d{d+2}}e^{i(x\xi_n + th_n^2|\xi_n|^4 + t|\xi_n|^2)}e^{-i|\xi_n|^2t\Delta}[g_n(\phi\circ \ell_{n})](\ell_{n}^{-1}(x+4th_n^2|\xi_n|^2\xi_n+2t\xi_n))\|_{L^{\frac{2(d+2)}d}_{t,x}} \\
&\qquad \leq 3^{-\frac 1{2(d+2)}}2^{-\frac d{2(d+2)}}B\|\phi\|_{L^2_x}.
\end{align*}

In addition, by the Strichartz inequality \eqref{E:Strichartz} for 4th order Schr\"odinger, the operator 
$$
\phi \mapsto D_\mu^{\frac d{d+2}} S_\mu(t)g_n(e^{i(\cdot)a_n} \phi)
$$ 
is also uniformly bounded from $L^2_x$ to $L^{\frac{2(d+2)}d}_{t,x}$, so it suffices to prove the lemmas when $\phi$ is in some dense subset of $L^2_x$.  Thus we may assume that $\phi$ is a Schwartz function with compact frequency support that does not contain 0:
$$
\supp \hat \phi \subseteq \{R^{-1} \leq |\xi| \leq R\}.
$$

Under the hypotheses of Lemma~\ref{L:limit S0}, we assume $|a_n| \geq 2R$ and $|\xi-a_n| \leq R$.  Then $|\xi| \sim |a_n|$, so by the fundamental theorem of calculus,
$$
\bigl|\tfrac{|a_n|^{\frac d{d+2}}}{|\xi|^{\frac d{d+2}}}-1\bigr| \sim |a_n|^{-\frac d{d+2}}\bigl| |a_n|^{\frac d{d+2}}-|\xi|^{\frac d{d+2}}\bigr| \lesssim |a_n|^{-\frac d{d+2}}|a_n|^{-\frac 2{d+2}}R.
$$
Since the function $e^{i(\cdot) a_n}\phi$ has frequency support on $\{|\xi-a_n| \leq R\}$, 
$$
\left\|e^{i(\cdot)a_n} \phi - |a_n|^{\frac{d}{d+2}} D_\mu^{-\frac{d}{d+2}} e^{i(\cdot)a_n}\phi\right\|_{L^2_x} \lesssim |a_n|^{-1}R \to 0.
$$
Therefore, since $D_\mu^{\frac{d}{d+2}} S_\mu(t)$ is a bounded operator from $L^2_x$ to $L^{\frac{2(d+2)}d}_{t,x}$, \eqref{E:limit S0} would follow from 
\begin{equation} \label{E:remove nabla}
|a_n|^{\frac d{d+2}}\|e^{it\Delta^2}(e^{i(\cdot)a_n} \phi) -  e^{i(x \cdot a_n + t(|a_n|^4-|a_n|^2 \Delta))} (\phi \circ \ell_{a_n})(\ell_{a_n}^{-1}(x+4t|a_n|^2a_n))\|_{L^{\frac{2(d+2)}d}_{t,x}}  \to 0.
\end{equation}
Changing variables in $t$, the left hand side of \eqref{E:remove nabla} equals
$$
\|e^{i\tfrac t{|a_n|^2}\Delta^2}e^{i(\cdot)a_n}\phi - e^{i(x a_n + t|a_n|^2)}e^{-it\Delta}(\phi \circ \ell_{a_n})(\ell_{a_n}^{-1}(x+4a_nt))\|_{L^{\frac{2(d+2)}d}_{t,x}}.
$$

Next, we compute
\begin{align*}
&e^{i\tfrac{t}{|a_n|^2}\Delta^2} e^{i(\cdot)a_n} \phi(x) = \int e^{i(x\xi + \frac{t}{|a_n|^2}|\xi|^4)}\hat \phi(\xi-a_n)\, d\xi
 = \int e^{ix(\xi+a_n)}e^{i\frac t{|a_n|^2}|\xi+a_n|^4} \hat \phi(\xi)\, d\xi\\
&\qquad= e^{i(x \cdot a_n + t|a_n|^2)} \int e^{i(x+4ta_n)\xi} e^{it(\frac{|\xi|^4}{|a_n|^2} + \frac{4|\xi|^2\xi a_n}{|a_n|^2} + 2|\xi|^2 + \frac{4(a_n \xi)^2}{|a_n|^2})} \hat \phi(\xi)\, d\xi\\
&\qquad= e^{i(x a_n + t|a_n|^2)} \int e^{i(x + 4ta_n )\xi} e^{it(\frac{|\xi|^4}{|a_n|^2} + \frac{4|\xi|^2\xi a_n}{|a_n|^2} + |\ell_{a_n}(\xi)|^2)} \hat \phi(\xi)\, d\xi.
\end{align*}

Thus \eqref{E:remove nabla} would follow from 
\begin{equation} \label{E:drop translations 0}
\|e^{it(\frac{\Delta^2}{|a_n|^2} + \frac{4i\Delta \nabla a_n}{|a_n|^2} + |\ell_{a_n}(-i\nabla)|^2)} \phi(x) - e^{-it\Delta}(\phi \circ \ell_{a_n}) \circ \ell_{a_n}^{-1}(x)\|_{L^{\frac{2(d+2)}d}_{t,x}} \to 0.
\end{equation}
Since  
\begin{equation} \label{E:deconjugate}
(e^{-it\Delta}(\phi \circ \ell_{a_n})) \circ \ell_{a_n}^{-1} = e^{it|\ell_{a_n}(-i\nabla)|^2}\phi ,
\end{equation}
the limit in \eqref{E:drop translations 0} equals zero if and only if
\begin{equation} \label{E:drop trans conj}
\lim_{n \to \infty} \|[e^{it(\frac{\Delta^2}{|a_n|^2} + \frac{4i\Delta \nabla a_n}{|a_n|^2})}-1] e^{it|\ell_{a_n}(-i\nabla)|^2} \phi\|_{L^{\frac{2(d+2)}d}_{t,x}} = 0, \qtq{if} |a_n| \to \infty.
\end{equation}

Similar computations show that \eqref{E:S1 to S0}, \eqref{E:S1 to Sch 0}, \eqref{E:S1 to Sch 1}, and \eqref{E:S1 to Sch infty} would (respectively) follow from 
\begin{gather} \label{E:S1 to S0'}
\lim_{n \to \infty} \|[e^{-ith_n^2\Delta}-1]e^{it\Delta^2}\phi\|_{L^{\frac{2(d+2)}d}_{t,x}} = 0, \qtq{if} h_n \to 0,
\\ \label{E:S1 to Sch 0'}
\lim_{n \to \infty} \|[e^{i\frac t{h_n^2}\Delta^2}-1]e^{-it\Delta}\phi\|_{L^{\frac{2(d+2)}d}_{t,x}} = 0, \qtq{if} h_n \to \infty,
\\ \label{E:S1 to Sch 1'}
\lim_{n \to \infty} \|[e^{it(\frac{\Delta^2}{h_n^2}+\frac{4i\Delta\nabla\xi_n}{h_n})}-1]e^{it|\tilde\ell_{\xi_n}(-i\nabla)|^2}\phi\|_{L^{\frac{2(d+2)}d}_{t,x}} = 0, \qtq{if} |\xi_n| \lesssim 1,\: h_n \to \infty,
\\ \label{E:S1 to Sch infty'}
\lim_{n \to \infty} \|[e^{it(\frac{\Delta^2}{|h_n\xi_n|^2} + \frac{4i\Delta\nabla\xi_n}{|h_n\xi_n|}-\frac{\Delta}{|\xi_n|^2})}-1]e^{it|\ell_{\xi_n}(-i\nabla)|^2}\phi\|_{L^{\frac{2(d+2)}d}_{t,x}} = 0, \qtq{if} |\xi_n|, |h_n\xi_n| \to \infty.
\end{gather}

Since $\hat\phi$ is smooth with compact support, 
$$
[e^{it(\frac{\Delta^2}{|a_n|^2} + \frac{4i\Delta \nabla a_n}{|a_n|^2}-\frac{\mu\Delta}{|a_n|^2})}-1] e^{it|\ell_{a_n}(-i\nabla)|^2} \phi \to 0
$$
pointwise in $t,x$, so \eqref{E:drop trans conj} will follow from the dominated convergence theorem if we show that there exists a function that dominates each term in the sequence.

Let 
\begin{align*}
\Phi_n(t,x,\xi) &= t(\tfrac{|\xi|^4}{|a_n|^2} + \tfrac{4|\xi|^2 \xi a_n}{|a_n|^2}+\tfrac{\mu|\xi|^2}{|a_n|^2} + |\ell_{a_n}(\xi)|^2) + x\xi,\\
\Psi_n(t,x,\xi) &= t|\ell_{a_n}(\xi)|^2 + x\xi.
\end{align*}
Then the left hand side of \eqref{E:drop trans conj} is the $L^{\frac{2(d+2)}d}_{t,x}$ norm of 
$$
(t,x) \mapsto \int[e^{i\Phi_n(t,x,\xi)} - e^{i\Psi_n(t,x,\xi))}]\hat\phi(\xi)\, d\xi.
$$
Since $\hat\phi$ is smooth with compact support, this quantity is uniformly bounded.  The gradients of the phases are
\begin{align*}
\nabla_\xi \Phi_n(t,x,\xi) &= t(\tfrac{4|\xi|^2\xi}{|a_n|^2} + \tfrac{8\xi\cdot a_n \xi}{|a_n|^2} + \tfrac{4|\xi|^2a_n}{|a_n|^2}+\tfrac{2\mu\xi}{|a_n|^2} + 2 \nabla \ell_{a_n}\circ\ell_{a_n}(\xi)) + x,\\
\nabla_\xi \Psi_n(t,x,\xi) &= 2t \ell_{a_n}\circ\ell_{a_n}(\xi) + x,
\end{align*}
and for $|a_n| \gg (1+R)^2$ and $|x| > 100 Rt$, these are nonvanishing for $\xi \in \supp \hat\phi \subset \{\frac 1R\le |\xi| \leq R\}$.  In particular,
\begin{equation} \label{E:gradients}
|\nabla_\xi \Phi_n(t,x,\xi)|, |\nabla_\xi \Psi_n(t,x,\xi)| \gtrsim |x|, \qquad |x| > 100 Rt.
\end{equation}
Furthermore, the Hessian matrices of the phases satisfy
\begin{equation} \label{E:hessians}
D^2_\xi \Phi_n(t,x,\xi), D^2_\xi \Psi_n(t,x,\xi) = t(O_{\xi,a_n}(\tfrac{R^2}{|a_n|^2}) + 2 \ell_{a_n}^2),
\end{equation}
where $O_{\xi,a_n}(\tfrac{R^2}{|a_n|^2})$ is a matrix whose coefficients are uniformly bounded by $\tfrac{R^2}{|a_n|^2}$, and we are identifying the transformation $\ell_{a_n}$ with its matrix.  Since $2\ell_{a_n}^2$ is uniformly positive definite (indeed, its eigenvalues are $12, 4,\ldots,4$), for $|a_n|$ sufficiently large (depending on $R$), the critical points of the phases must be nondegenerate.  Thus by \eqref{E:gradients}, \eqref{E:hessians}, and the principle of stationary phase (cf.\, Stein \cite[Ch.~VIII]{Stein:1993}),
$$
|\int[e^{i\Phi_n(t,x,\xi)} - e^{i\Psi_n(t,x,\xi))}]\phi(\xi)\, d\xi| \lesssim \bigl(\tfrac1{1+|x|}\bigr)^N\chi_{\{|x| \gg t\}} + \bigl(\tfrac1{1+|t|}\bigr)^{\frac d2} \chi_{\{|x| \lesssim t\}} \lesssim \bigl(\tfrac1{1+|t|+|x|}\bigr)^{\frac{d}2}.
$$
The right hand side of the above inequality is in $L^{\frac{2(d+2)}d}_{t,x}$, so \eqref{E:drop trans conj} does indeed follow by the dominated convergence theorem.  

The limits in \eqref{E:S1 to S0'}, \eqref{E:S1 to Sch 0'}, \eqref{E:S1 to Sch 1'}, and \eqref{E:S1 to Sch infty'} may be verified in a similar manner.  (For \eqref{E:S1 to S0'} and \eqref{E:S1 to Sch 1'}, one also uses that $0 \notin \supp \hat\phi$.)  This completes the proof of Lemmas~\ref{L:limit S0} and~\ref{L:limit S1}.  
\end{proof}


\end{document}